\documentclass[12pt]{amsart} 
\usepackage{amssymb}
\usepackage{amsmath}
\usepackage{latexsym}
\usepackage{amscd}
\usepackage{amsfonts}
\usepackage{pb-diagram}
\usepackage[all]{xypic}
\usepackage[mathcal,mathscr]{eucal}
\usepackage{amsthm}
\usepackage{graphicx}
\usepackage{url}
\usepackage{pdfsync}
\usepackage{bbm}
\usepackage{xypic}
\usepackage[pdftex,colorlinks]{hyperref}
\usepackage{pifont}
\usepackage{amsbsy}
\usepackage{epstopdf}
\usepackage{times}
\usepackage{epic,eepic}

\newtheorem{theorem}{Theorem}[section]
\newtheorem{prop}[theorem]{Proposition}
\newtheorem{lemma}[theorem]{Lemma}
\newtheorem{corollary}[theorem]{Corollary}

\newtheorem{definition}[theorem]{Definition}

\theoremstyle{remark}
\newtheorem{remark}[theorem]{\bf {Remark}}
\newtheorem{example}[theorem]{\bf {Example}}

\numberwithin{equation}{section}

\DeclareMathOperator{\cpt}{cpt}
\DeclareMathOperator{\Pf}{Pf}
\DeclareMathOperator{\dR}{dR}

\DeclareMathOperator{\Or}{or}

\DeclareMathOperator{\End}{End}

\DeclareMathOperator{\Ker}{Ker}

\DeclareMathOperator{\PD}{PD}

\DeclareMathOperator{\supp}{supp}

\DeclareMathOperator{\tina}{\widetilde{\nabla}}

\DeclareMathOperator{\sing}{sing}

\DeclareMathOperator{\SO}{SO}

\DeclareMathOperator{\Bl}{{\bf B}l}

\DeclareMathOperator{\rank}{rank}

\DeclareMathOperator{\id}{id}
\DeclareMathOperator{\inte}{int}

\DeclareMathOperator{\TPf}{TPf}

\DeclareMathOperator{\Imag}{Im}

\newcommand\bR{{\mathbb R}}

\newcommand{\eE}{\EuScript{E}}

\newcommand{\eT}{\EuScript{T}}

\newcommand{\eSt}{\EuScript{S}}

\newcommand{\ra}{\rightarrow}

\begin{document}

\title{Chern-Gauss-Bonnet  and Lefschetz duality  from a currential point of view}

\author{Daniel Cibotaru}
\address{Universidade Federal do Cear\'a, Fortaleza, CE, Brasil}
\email{daniel@mat.ufc.br} 

\subjclass[2010]{Primary 58A25,  49Q15; Secondary 53C05.} 
\begin{abstract} We  use the mapping cone  for the relative deRham cohomology of a manifold with boundary in order to show that the  Chern-Gauss-Bonnet Theorem  for oriented Riemannian vector bundles over such manifolds is a manifestation of Lefschetz Duality in any of the two embodiments of the latter. We explain how  Thom isomorphism fits into this picture, complementing thus the classical results about Thom forms with compact support. When the rank is  \emph{odd}, we construct, by using secondary transgression forms introduced here, a new  closed pair of forms on the disk bundle associated to a vector bundle, pair  which is Lefschetz dual to the zero section.

\end{abstract}
\maketitle
\tableofcontents
\section{Introduction}

The mapping cone construction is a standard tool in algebraic topology. The purpose of this note is to put it to good use   in the context of  relative deRham cohomology on a manifold with boundary.  So far, the Dirichlet representation of relative deRham classes as forms which pull-back to zero over the boundary has been by far the favorite sister in the literature. Nevertheless, we contend that working with (closed) pairs of forms has certain advantages such as turning explicit several maps of interest in cohomology. Moreover, the mapping cone construction, which any student can learn about for example from the classical book of Bott and Tu \cite{BT} sheds new light on certain classical results like the Chern-Gauss-Bonnet Theorem on manifolds with boundary. The fact that Chern Theorem \cite{Ch2} on such manifolds is indeed a manifestation of Lefschetz Duality will not come as a surprise for anybody who knows that its boundaryless counterpart is a combination of two facts:  an explicit Poincar\'e Duality statement plus the Poincar\'e-Hopf Theorem. It might also be expected to get in fact two  statements  that can be called Chern-Gauss-Bonnet Theorem for a manifold with boundary corresponding to the two classes of isomorphisms:
\begin{center}
relative cohomology $\longleftrightarrow$ absolute homology\\~\\
absolute cohomology $\longleftrightarrow$ relative homology
\end{center}
What we found rather surprising is that one can cast the Chern-Gauss-Bonnet Theorem on a manifold with boundary  also as a realization of  Thom isomorphism. The closed pair (Pfaffian, transgression of the Pfaffian) plays the same role as the Thom form with compact support when one considers  on one side of the Thom isomorphism the  relative cohomology of the pair (disk bundle, spherical bundle) instead of the (isomorphically equivalent) cohomology with compact supports. This is the \emph{even} rank picture. Now Thom isomorphism holds irrespective of the parity of the rank, provided the vector bundle is oriented. Therefore,  when the rank is \emph{odd} we produce a new closed pair of forms on the same manifold with boundary that fullfills the same property as the already described pair in the even case. In particular, this pair is Lefschetz dual to the zero section. This result could probably be  viewed as an odd rank Chern-Gauss-Bonnet Theorem for a manifold with boundary.  In both cases, the  pairs can be used to define \emph{explicit} Thom forms with compact support in a straighforward manner.

Before we take a look at  some details let us say a few words about the techniques used to prove these results. Following the work of Harvey and Lawson  and their school \cite{HL1,La}, we presented in \cite{Ci}  a general transgression formula for vertical, tame Morse-Bott-Smale flows over fiber bundles $P\ra B$ and one of the applications included in \cite{Ci} was a short proof of the Chern-Gauss-Bonnet. The proof of the general transgression formula is a rather intricate business, but the context in which it is used for the proof of the Chern Theorem elicits a more straightforward approach, extendable also to the case of manifolds with boundary. Therefore what we do here is pretty much self-contained, using the previously cited works only for inspiration. We could add that the most sophisticated tool in this note is probably Stokes Theorem.

The main results are as follows. Let $M$ be an oriented compact manifold with boundary of dimension $n$ and let
\[ \Omega^k(M,\partial M):=\Omega^k(M)\oplus \Omega^{k-1}(\partial M)\]
be the vector space of pairs of differential forms. Define a differential operator on pairs:
\[ d(\omega,\gamma)=(-d\omega, \iota^*\omega+d\gamma),
\]
where $\iota:\partial M\ra M$ is the canonical inclusion. It is well-known that the resulting cohomology groups are isomorphic with the relative singular cohomology groups of $(M,\partial M)$. 

On the homology side, we will use currents with $\mathscr{D}'(\cdot)$ denoting the space of distributionally valued forms, another name for currents. Define:
\[ \mathscr{D}'_k(M,\partial M):=\mathscr{D}'_k(M)\oplus\mathscr{D}'_{k-1}(\partial M)
\]
with differential:
\[ d(T,S)=(\iota_*S-dT,dS).
\]
The homology of this complex is the relative homology of $M$. Then Lefschetz Duality takes the following form:
\pagebreak
\begin{theorem}\label{LefDuintro}
The maps:
\begin{itemize}
\item[(i)] \[\mathscr{L}_I:\Omega^*(M,\partial M)\ra\mathscr{D}'_{n-*}(M),\]\[ \mathscr{L}_I(\omega,\gamma)= \left\{\eta\ra \int_M\omega\wedge \eta +\int_{\partial M}\gamma\wedge\iota^*\eta\right\}  \]
\item[(ii)] \[\mathscr{L}_{II}:\Omega^*(M)\ra\mathscr{D}'_{n-*}(M,\partial M),\]\[ \mathscr{L}_{II}(\eta)= \left ( \omega\ra \int_M\omega\wedge \eta , \gamma\ra \int_{\partial M}\gamma\wedge\iota^*\eta\right)  \]
\end{itemize}
commute (up to a sign) with the differentials and induce isomorphisms in (co)homology.
\end{theorem}

We use these Lefschetz Duality isomorphisms in order to prove the following generalization of \cite{Ch2}:
\begin{theorem}\label{CGBintro} Let $\pi:E\ra B$ be an oriented Riemannian vector bundle over a manifold with boundary of rank $2k$. Suppose $E$ is endowed with a metric compatible connection $\nabla$ and a section $s:B\ra E$ transversal to the zero section. Then 
\begin{itemize}
\item[(a)] $\mathscr{L}_{II}(\Pf(\nabla))$ and $s^{-1}(0)$ represent the same class in $H_{n-2k}(B,\partial B)$, i.e. $s^{-1}(0)$ and the Pfaffian associated to $\nabla$ are Lefschetz dual;
\item[(b)] If $s\bigr|_{\partial B}$ is everywhere non-vanishing then there exists a transgression class $\TPf(\nabla,s)\in \Omega^{2k-1}(\partial M)$ such that the pair $(\Pf(\nabla),-\TPf(\nabla,s))$ is closed and \linebreak $\mathscr{L}_{I}(\Pf(\nabla),-\TPf(\nabla,s))$ and $s^{-1}(0)$ represent the same class in $H_{n-2k}(B)$.  
\end{itemize}
\end{theorem}
The standard Chern Theorem is obtained when $E=TB$ in part (b) of the above Theorem as Poincar\'e-Hopf Theorem is known to hold on manifolds with boundary as well.

 We then turn to analize the following  diagram of isomorphisms:
 \begin{equation}\label{eqintro} \xymatrix{ H^{i+k}(E, E_0)\ar[r]^{\iota^*}\ar[d]_{\mu} & H^{i+k}(\overline{DE},SE)\ar[r]^{LD ~ I}\ar@/_0.5pc/[d]_{\nu} & H_{n-i}(\mathscr{D}_*(\overline{DE}))\ar@/^0.5pc/[d]^{\pi_*}\\
             H^{i+k}_{\cpt}(E) \ar@/^0.5pc/[r]^{\int} &  H^{i}(B)\ar@/_0.5pc/[u]_{\nu^{-1}}\ar@/^0.5pc/[l]^{\tau \wedge\pi^*( \cdot)}\ar[r]^{PD\qquad} &  H_{n-i}(\mathscr{D}_*(B))\ar@/^0.5pc/[u]^{\iota_*}.}
\end{equation}
where $DE\ra B$ is the unit disk bundle associated to $E\ra B$. The bottom left corner is the well-known Thom isomorphism using forms with compact support. One of the main contributions of this note is to make explicit the first two vertical arrows in the above diagram in both the even and odd rank cases. The middle one is especially interesting.

Let $s^{\tau}:SE\ra \pi^*E$ be the tautological section.  As a consequence of Theorem \ref{CGBintro} we show:
\begin{theorem} \label{Th1.3}The pair $(\Pf(\pi^*\nabla),-\TPf(\pi^*\nabla,s^{\tau}))\in H^{2k}(\overline{DE},SE)$ is Lefschetz dual to the zero section $B\hookrightarrow \overline{DE}$ and consequently the isomorphism $\nu^{-1}$ is
\[ \eta\ra (\Pf(\pi^*\nabla)\wedge\pi^*\eta,-\TPf(\pi^*\nabla,s^{\tau})\wedge \pi^*\eta),
\]
with inverse:
\[ (\omega,\gamma)\ra \int_{\overline{DE}/B}\omega+\int_{SE/B}\gamma.
\]
\end{theorem}
 
 When the rank of $\pi:E\ra B$ is  $2k-1$, the following construction can be used. Consider the bundle $\bR \oplus\pi^*E\ra S(\bR\oplus E)$.\footnote{We denote by $\pi$ all the projections of the fiber bundles $E$, $\overline{DE}$, $SE$ and $S(\bR\oplus E)$ to $B$.} It comes with two non-vanishing sections, the obvious one $(1,0)$ and the tautological one $s^{\tau}$. Taking the transgression of $\nabla^1:=d\oplus \pi^*\nabla$ and $\nabla^2:=\nabla^{\tau}\oplus \nabla^{\tau^{\perp}}$, where the later connection is determined by orthogonally projecting $\pi^*\nabla\oplus d$ onto the tautological bundle and its complement, we get a closed form $\TPf(\nabla^1,\nabla^2)\in \Omega^{2k-1}(S(\bR\oplus E))$. By using the inverse of the stereographic projection in every fiber we can "transfer" it to a closed form $\TPf(\nabla^1,\nabla^2)\in\Omega^{2k-1}(\overline{DE})$.  
 
 We contend that along the equator $SE\hookrightarrow S(E\oplus\bR)$, the form $\TPf(\nabla^1,\nabla^2)$ is exact. In fact, along the equator the sections $(1,0)$ and $s^{\tau}$ are orthogonal and hence they span a \emph{trivializable} plane bundle $\mathscr{P}$ at every point. We therefore get another splitting of $\bR \oplus\pi^*E\bigr|_{SE}$ into $\mathscr{P}$ and its orthogonal complement and therefore a third connection $\nabla^3$ resulting by taking $d\oplus \nabla^{\mathscr{P}^{\perp}}$ on $\mathscr{P}\oplus \mathscr{P}^{\perp}$. Here, the connection $\nabla^{\mathscr{P}^{\perp}}$ results by orthogonally projecting $d\oplus \pi^*\nabla$. 
 
 We use the available three connections to produce a secondary transgression class \linebreak $\TPf(\nabla^1,\nabla^2,\nabla^3)\in \Omega^{2k-2}(SE)$ such that 
 \[ d\TPf(\nabla^1,\nabla^2,\nabla^3)=-\TPf(\nabla^1,\nabla^2).
 \]
 The way this works is by considering the $2$-dimensional simplex $\Delta^2$ and the auxiliary connection $\tilde{\nabla}$ on the bundle $\bR\oplus p_2^*E\ra \Delta^2\times B$:
 \[ \tilde{\nabla}=\frac{d}{ds}+\frac{d}{dt}+\nabla^1+s(\nabla^2-\nabla^1)+t(\nabla^3-\nabla^1).
 \]
 We denoted $p_2:\Delta^2\times B\ra B$ the obvious projection. Then by definition:
 \[ \TPf(\nabla^1,\nabla^2,\nabla^3)=\int_{\Delta^2}\Pf(\tilde{\nabla}).
 \]
 The main result of this note is:
 \begin{theorem}\label{Th1.4} If $E\ra B$ has odd rank, then the pair $(-\TPf(\nabla^1,\nabla^2), -\TPf(\nabla^1,\nabla^2,\nabla^3))\in H^{2k-1}(\overline{DE},SE)$ is Lefschetz dual to the zero section $B\hookrightarrow \overline{DE}$ and consequently the isomorphism $\nu^{-1}$ in this case is
 \[  \eta\ra (-\TPf(\nabla^1,\nabla^2)\wedge \pi^*\eta,-\TPf(\nabla^1,\nabla^2,\nabla^3)\wedge \pi^*\eta).
 \]
 \end{theorem}
 
 When combining Theorems \ref{Th1.3} and \ref{Th1.4} with the explicit form of the isomorphism $\mu$ one gets concrete representatives for Thom forms with compact support (and consequently concrete Poincar\'e duals for compact oriented submanifolds), both in the even and in the odd cases (see Corollary \ref{Niccor} for details) a feature that was already known to Nicolaescu \cite{Ni} in the even case. It is necessary mentioning that this type of result is also a byproduct of the theory of singular connections in \cite {HL2} developed by Harvey and Lawson.

 We comment about  connections  to previous work. The proof of Theorem \ref{LefDuintro} reduces easily to the statement that the pairing
 \[  \Omega_D^k(M,\partial M)\times \Omega^{n-k}(M),\qquad (\omega,\eta)\ra \int_M\omega\wedge\eta,
 \] 
 descends to a non-degenerate bilinear map in cohomology. We used  $\Omega_D(M,\partial M)$ to denote forms in $M$ whose pull-back to $\partial M$ is zero. This statement seems to have been known for a long time. Nevertheless, the reader will find here two proofs, one that uses the Hodge-Morrey-Friedrichs decomposition as in \cite{Sch} and one which is more elementary.
 
 In  Section 6 of \cite{BC}, Bott and Chern give a slight extension of the result from \cite{Ch2} to the case when the  rank of the bundle equals the dimension of the manifold (see Proposition 6.3 from \cite{BC}). This is  a particular case of Theorem \ref{CGBintro} (b) above.
 
 A form of Lefschetz-deRham Duality  was considered by Harvey and Lawson in \cite{HL1} as a byproduct of their analysis  
 on flowing forms via Morse-Stokes vector fields. Their main result in this direction is a homotopy equivalence between the (Dirichlet) relative deRham complex and a certain subcomplex of $\mathscr{D}'_*$ generated by a finite number of rectifiable submanifolds (stable submanifolds of the flow). They  again use the Dirichlet-deRham complex \cite{HL3} when  extending their own work on Federer-deRham differential characters to manifolds with boundary. There are several differences between \cite{HL1,HL3} and what is exposed here. The main novelty of our approach  seems to be the systematic use of the mapping cone both at the level of cohomology as well as at the level of homology. 
 
  The proofs of Theorem \ref{CGBintro} and of Theorem \ref{Th1.4} rely on computing certain limits of $1$-parameter families of forms in the flat topology of currents. Notice that the way we do that here (e.g. the demonstration of Theorem \ref{transfor})  can be traced back to the techniques that Harvey and Lawson  introduced.

 In \cite{Ni}, Nicolaescu gives a proof of the Chern-Gauss-Bonnet Theorem for manifolds without boundary in which he uses an explicit Thom form with compact support. The reader can encounter some of those ideas in the construction of the isomorphism $\mu$\footnote{For the definition of $\mu$ see (\ref{eq6.4}).} of (\ref{eqintro}) and  as a leitmotiv for Theorems \ref{Th1.3} and  \ref{Th1.4}. 
 
 We would like to highlight one more point.  We have paid special attention to signs in this note and painstainkingly worked to choose the "best fitting" conventions.  This is one reason we preferred to verify in minute detail certain claims as for example describing all arrows in the commutative ladders (\ref{diag2}) and (\ref{eq6.2}).  
 
 \smallskip
 \noindent
   \emph{Acknowledgements:} I would like to thank Vincent Grandjean, Luciano Mari,  and Diego Moreira for many enlightening discussions. I am also grateful for the  working environment and support in the Mathematics Department of the UFC, made possible by Greg\'orio Pacelli Bessa, Jorge de Lira and  Eduardo Teixeira, among others.

  \section{Lefschetz-de Rham Duality }
We start with some basic facts. Let $M$ be an oriented manifold of dimension $n$ without boundary. Let $(\Omega^*(M),d)$ be the (co)chain complex of smooth forms of degree $k$ on $M$.  

 The deRham version of  Poincar\'e theorem  presents itself as an isomorphism:
\begin{equation}\label{eq11} H^{n-k}(\Omega^*(M))\simeq (H^k(\Omega^*_{\cpt}(M)))^*,
\end{equation}
where the vector space $\Omega^*_{\cpt}(M)$ is the  chain complex  of compactly supported smooth forms. We will use the standard notation $H^*(M), \; H^*_{\cpt}(M)$ for the groups appearing in (\ref{eq11}).  The isomorphism is induced by the following well-defined, non-degenerate, bilinear pairing:
\[ H^{n-k}(M)\times H_{\cpt}^k(M)\ra \bR,\qquad ([\omega],[\eta])\ra \int_M\omega\wedge\eta.
\]
One can interpret the group $(H_{\cpt}^k(M))^*$ as a homology group is as follows. Let $(\mathscr{D}'_*(M),d)$ be the dual chain complex, where $\mathscr{D}'_k$,  the space of currents of dimension $k$, is the topologically dual vector space of $\Omega^k_{\cpt}(M)$. The differential $d$\footnote{We prefer $d$ rather than $\partial$ to avoid confusion later due to the overuse of the latter symbol.} for currents is:
\[ d:\mathscr{D}'_{k}(M)\ra \mathscr{D}'_{k-1}(M),\qquad dT(\eta):=T(d\eta).
\]
We denote by $H_k(M,\mathscr{D}_*')$ the $k$-th homology group of this chain complex. A short argument\footnote{Repeat the argument in Lemma \ref{HMD}.} based on the Hahn-Banach Theorem shows that there exist a natural isomorphism $(H_{\cpt,\dR}^k(M))^*\simeq H_k(M,\mathscr{D}_*')$. Poincar\'e Duality is therefore :
\[ H^{n-k}_{\dR}(M)\simeq H_k(M,\mathscr{D}_*').
\] 
Now deRham isomorphism implies that for a compact manifold, $H_k(M,\mathscr{D}_*')$ is isomorphic indeed with the singular $k$-th homology group of $M$ with real coefficients. In the non-compact case, $H_k(M,\mathscr{D}_*')$ is the Borel-Moore homology (see \cite{BM}).

On compact manifolds with boundary $M$, the counterpart of Poincar\'e duality is Lefschetz duality (see \cite{Br}). This time one has two dualities: 
 \[ H^{k}_{\sing}(M,\partial M)\simeq H_{n-k,\sing}(M),\qquad H^k_{\sing}(M)\simeq H_{n-k,\sing}(M,\partial M)\]
Moreover, up to sign, the following ladder is commutative up to some universal signs (see Theorem 9.2 in \cite{Br}):
 \[ \xymatrix{\ar[r]^{\qquad\;}&H^k(M) \ar[r]^{}\ar[d] \ar@{}| {}^{\quad}[dr] & H^k(\partial M) \ar[r]^{\;\;\;\;} \ar[d]  \ar@{}| {}^{\;\; \;\;}[dr]& H^{k+1}(M,\partial M)\ar[r]^{\;\;\;} \ar[d] \ar@{}| {\;\;}^{\;\; \;\;\;\;}[dr] & H^{k+1}(M) \ar[d]\ar[r]^{\qquad}&\\
\ar[r]^{\qquad\quad}&H_{n-k}(M,\partial M)\ar[r]^{} &H_{n-k-1}(\partial M) \ar[r]^{} & H_{n-k-1}(M)\ar[r]^{\quad} & H_{n-k-1}(M,\partial M)\ar[r]^{\qquad \;\;\;\;}& }
 \]
The vertical arows are the duality isomorphisms. All groups represent singular homology/cohomology.

 If one turns to a deRham point of view, there are more than two ways  for introducing $H^k(M,\partial M)$. First, one has:
\[ \Omega^k_D(M, \partial M):=\{\omega\in\Omega^k(M)~|~\omega\bigr|_{\partial M}\equiv 0\}.
\]
The subscript $D$ is meant to suggest Dirichlet boundary conditions as they are  usually called (\cite{Sch}). It is easy to infer from the short exact sequence of chain complexes
\[ 0\ra \Omega^*_D(M,\partial M)\ra \Omega^*(M)\stackrel{\iota^*}{\ra} \Omega^*(\partial M)\ra 0
\]
that $H^k( \Omega^*_D(M,\partial M))$ plays indeed the role of relative cohomology. 
\begin{remark}
 It is also not  too hard to see that the topological dual of this space is isomorphic with $\mathscr{D}_k'(M)/\mathscr{D}_k'(\partial M)$, a space seemingly not too friendly to work with.  \end{remark}
 
 Second, one can define the relative deRham cohomology groups $H^k_{\dR}(M,\partial M)$ as the cohomology groups of the homological mapping cone for the pull-back map \linebreak $\iota^*:\Omega^*(M)\ra \Omega^*(\partial M)$. This is the chain complex:
 \[\Omega^{*}(M,\partial M):=\Omega^{*}(M)\oplus \Omega^{*-1}(\partial M)
 \]
 with differential
\begin{equation}\label{deq1}\Omega^{k}(M,\partial M)\ra \Omega^{k+1}(M,\partial M),\qquad (\omega,\gamma)\ra (-d\omega, \iota^*\omega+d\gamma)
\end{equation}
The functional dual to $\Omega^k(M,\partial M)$ is  another space of pairs:
 \[\mathscr{D}'_k(M,\partial M)=\mathscr{D}'_k(M)\oplus\mathscr{D}'_{k-1}(\partial M)\] 
 We define a  differential on it that turns it into a chain complex:
\begin{equation}\label{deq2} d(T,S)=(\iota_*S-dT,dS).
\end{equation}
\begin{example}
Notice that, with this definition of chain differential, an oriented, compact submanifold $N$ with boundary $\partial N\subset \partial M$ considered as a pair $(N,\partial N)\hookrightarrow (M,\partial M)$ naturally determines a closed pair
\[([N], [\partial N])\in \mathscr{D}'_k(M, \partial M ),\]
 where $k=\dim{N}$.
 \end{example}
 \begin{remark}\label{pd} One can define in a similar manner a chain complex $\mathscr{D}'_*(M,N)$ for any manifold pair $(M,N)$ and indeed for any smooth map $\iota:N\ra M$ not necessarily an embedding. For a smooth map of pairs $A:(M_1,N_1)\ra (M_2,N_2)$ define the push-forward:
 \[ A_*(T,S):=(A_*T,A_*S). 
 \]
 It is trivial to check that push-forward commutes with $d$.
 \end{remark}

Define now the following pairing:
\begin{equation}\label{pair1} \mathscr{L}:\Omega^k(M,\partial M)\times \Omega^{n-k}(M)\ra \bR,\qquad (\omega,\gamma;\eta)\stackrel{\mathscr{L}}{\ra} \int_M\omega\wedge\eta+\int_{\partial M} \gamma\wedge\iota^*\eta.
\end{equation}
Notice that this pairing induces  a continuous and injective map:
\[ \mathscr{L}_I:\Omega^k(M,\partial M)\ra \mathscr{D}_{n-k}'(M), \qquad \mathscr{L}_I(\omega,\gamma) = \{\eta\ra\mathscr{L}_{(\omega,\gamma)}(\eta)\};
\]
\begin{remark} In terms of operations with currents:
 \[\mathscr{L}_I(\omega,\gamma)=\omega+\iota_*\gamma,\]
 where on the r.h.s. one understands the push-forward of the current represented by $\gamma$.
 \end{remark}
 
In a similar vein one defines:
\[ \mathscr{L}_{II}:\Omega^{n-k}(M)\ra \mathscr{D}_{k}'(M,\partial M),\quad \mathscr{L}_{II}(\eta)=\left(\omega\ra \int_{M}\omega\wedge\eta\;;\;\gamma\ra \int_{\partial M} \gamma\wedge\iota^*\eta \right).
\]

The pairing (\ref{pair1}) is compatible with the chain differentials by which we mean that $\mathscr{L}_I$  commutes  with $d$ up to a sign. As a matter of fact,  Stokes Theorem  implies for $\deg{\omega}=k$:
\begin{equation}\label{eq2} \int_M-d\omega\wedge \eta+\int_{\partial M}(\iota^*\omega+d\gamma)\wedge\iota^*\eta=(-1)^{k}\left(\int_M\omega\wedge d\eta+\int_{\partial M} \gamma\wedge d\iota^*\eta\right)
\end{equation}
which proves that 
\begin{equation}\label{eq2'}\mathscr{L}_I(d(\omega,\gamma))=(-1)^{k}d\mathscr{L}_I(\omega,\gamma).\end{equation}
 We stress that we use the adjoint/dual of exterior differentiation without any sign attached to it as the differential on the spaces of currents $\mathscr{D}'_*(M)$. We break (\ref{eq2}) into two identities for $\deg{\eta}=n-k-1$.
 \begin{equation}\label{eq3}\int_{\partial M}\iota^*\omega\wedge\iota^*\eta- \int_Md\omega\wedge\eta =(-1)^{k}\int_M\omega\wedge d\eta \;\; \mbox{and}
 \end{equation}
 \begin{equation}\label{eq4} \int_{\partial M}d\gamma\wedge\iota^*\eta=(-1)^{k}\int_{\partial M}\gamma\wedge d\iota^*\eta
 \end{equation}
which together say that
\[ d\mathscr{L}_{II}(\eta)=(-1)^{k}\mathscr{L}_{II}(d\eta).
\]

A consequence of the compatibility relation  (\ref{eq2'}) is the good definition of  the pairing:
\begin{equation}\label{lefpa} H^k_{\dR}(M,\partial M)\times H^{n-k}_{\dR}(M)\ra \bR,\qquad ([\omega,\gamma],[\eta])\ra \mathscr{L}(\omega,\gamma;\eta)
\end{equation}
The following is a first example of Lefschetz Duality.

\begin{theorem}\label{LefDu} The pairing (\ref{lefpa}) is non-degenerate.
\end{theorem}
\begin{proof}  First, by Lemma \ref{drel} the map:
\[ \Omega_D^*(M, \partial M)\hookrightarrow \Omega^*(M,\partial M),\qquad \omega\ra (\omega,0) \]
induces an isomorphism in cohomology.  

Second, the pairing:
\[\overline{ \mathscr{L}}:H^k(\Omega_D^*(M,\partial M))\times H^{n-k}_{\dR}(M)\qquad \overline{\mathscr{L}}([\omega],[\eta])= \int_M\omega\wedge\eta
\]
is non-degenerate.  At this point,  the  Hodge-Morrey-Friederichs decomposition theorem on  a manifold with boundary as developed in \cite{Sch} proves useful.  In particular, Theorem 2.6.1 in \cite{Sch}  saya that every absolute deRham cohomology class in $M$ can be represented by a harmonic field with Neuman boundary conditions, i.e. a form $\omega$ which satisfies:
\[ d\omega=0,\;\;\quad d^*\omega=0,\;\;\quad  \iota_{\nu}\omega=0,
\]
where $\nu$ is the unit  exterior normal  and $\iota_{\nu}$ represents contraction. Similarl,y every relative deRham cohomology class in $H^k(\Omega_D^*(M,\partial M))$ can be represented by a harmonic field with Dirichlet boundary conditions\footnote{This is a third point of view on deRham relative cohomology.}, i.e. a form $\omega$ which satisfies:
\[ d\omega=0,\;\;\quad d^*\omega=0,\;\; \quad \iota^*\omega=0.
\]
Moreover the Hodge operator $*$ is an isomorphism between the two groups of harmonic fields (Corollary 2.6.2, which can be considered  yet another form of Lefschetz duality) and therefore if 
\[ \overline{ \mathscr{L}}([\omega],[\eta])=0\qquad \forall [\eta]
\]
then taking $\omega$ to be a harmonic field as above and $\eta=*\omega$ one gets that $\|\omega\|^2=0$ and hence $\omega=0$.  Analogously, if $\overline{ \mathscr{L}}([\omega],[\eta])=0$ for all $[\omega]$ then $[\eta]=0$.
\end{proof}
\begin{remark} In Section \ref{LDbcodim}, we give another proof of Theorem \ref{LefDu}, one that does not use the Hodge-Morrey-Friedrichs decomposition.
\end{remark}
\begin{lemma}\label{drel} The natural map:
\begin{equation}\label{eq1234} \Omega_D^*(M,\partial M)\hookrightarrow \Omega^*(M,\partial M),\qquad \omega\ra (\omega,0) \end{equation}
induces an isomorphism in cohomology.
\end{lemma}
\begin{proof} The map is  well-defined because it anti-commutes with $d$.

We prove  injectivity first. Suppose $(\omega,0)=(-d\eta,\iota^*\eta+d\gamma)$ for some pair \linebreak$(\eta,\gamma)\in \Omega^{k-1}(M,\partial M)$. By a partition of unity argument, the form $\gamma$ admits an extension $\tilde{\gamma}\in\Omega^{k-2}(M)$ such that
 \begin{equation}\label{eq123} \iota^*\tilde{\gamma}=\gamma.\end{equation} Let
\[ \tilde{\eta}:=\eta+d\tilde{\gamma}.
\]
Notice that $\iota^*\tilde{\eta}=\iota^*\eta+d\iota^*\tilde{\gamma}=0$. Moreover $d\tilde{\eta}=d\eta=-\omega$. Hence $\omega$ is exact in $\Omega^k_D(M,\partial M)$.

 We prove surjectivity. Let $(\eta,\gamma)\in\Omega^{k}(M,\partial M)$ be a closed pair. We use again a form $\tilde{\gamma}\in \Omega^{k-1}(M)$ which satisfies (\ref{eq123}). Notice that $(-d\tilde{\gamma},\iota^*\tilde{\gamma})=d(\tilde{\gamma},0)$ is an exact pair and
\[ (\eta,\gamma)-(-d\tilde{\gamma},\iota^*\tilde{\gamma})=:(\omega,0),
\]
with $\omega$ satisfying
 \[\iota^*\omega=\iota^*\eta +d\iota^*\tilde{\gamma}=-d\gamma+d\iota^*\tilde{\gamma}=0.\] 
\end{proof}

The following is a straightforward consequence of Theorem \ref{LefDu}:
\begin{corollary} The map:
\[ H^k_{\dR}(M,\partial M)\ra H_{n-k}(\mathscr{D}'_*(M)),\qquad [\omega,\gamma]\ra \mathscr{L}_I(\omega,\gamma)
\]
is an isomorphism.
\end{corollary}
\begin{proof} Use Theorem \ref{LefDu} and Lemma \ref{HMD} below.
\end{proof}
\begin{lemma}[Homological Duality Lemma]\label{HMD} The following holds:
\[ H_{n-k}(\mathscr{D}'_*(M))\simeq (H^{n-k}_{\dR}(M))^*.
\]
\end{lemma}
\begin{proof} The is based on the following more general fact, which we found in the online notes of B. Lawson and whose proof we include  for the convenience of the reader.

\vspace{0.3cm}

Let $A\stackrel{d}{\longrightarrow}B\stackrel{\delta}{\longrightarrow}C$ be continuous linear maps of Fr\^echet spaces such that $\delta\circ d=0$. Suppose $d$ and $\delta$  have closed range. Then there exists a natural \emph{algebraic} isomorphism:

\[ \frac{\Ker d^*}{\Imag \delta^*}=:H(B^*) \simeq H(B)^*:=\left(\frac{\Ker{\delta}}{\Imag d}\right)^*.
\]

The isomorphism comes from restricting the natural duality pairing  $B^*\times B\ra \bR$ to a pairing on the product $\Ker d^*\times \Ker \delta$ and noting that this pairing vanishes on both $\Ker d^*\times \Imag d$ and on $\Imag \delta^*\times \Ker \delta$, thus inducing a pairing:
\[ \frac{\Ker d^*}{\Imag \delta^*}\times \frac{\Ker \delta}{\Imag d}\ra \bR
\]
 and therefore a linear map $H(B^*)\ra H(B)^*$.
 
 Surjectivity of this map is a straightforward application of Hahn-Banach Theorem, by composing a continuous $\beta:\frac{\Ker \delta}{\Imag d}\ra \bR$ with the continuous projection $\Ker \delta\ra \frac{\Ker \delta}{\Imag d}$ and then extending the resulting map to  the entire $B$.
 
 In order to prove injectivity take $b^*\in B^*$ such that $b^*\bigr|_{\Ker \delta}\equiv 0$. We would like to prove that $b^*=\delta^* (c^*):=c^*\circ \delta$ for some $c^*\in C^*$. For that end, notice that $b^*$ induces a continuous map $\frac{B}{\Ker \delta}\ra \bR$ and that
 \[ \delta: \frac{B}{\Ker \delta}\ra \Imag \delta,
 \]
 is a linear homeomorphism by the Open Mapping Theorem which is valid also for Fr\^echet spaces (\cite{Tr}). This uses in a fundamental way the fact that $\Imag \delta$ is closed. Now $b^*\circ \delta^{-1}:\Imag \delta \ra \bR$ can be extended by the Hahn-Banach Theorem to a a continuous map $c^*:C \ra \bR$ and one checks that $b^*=c^*\circ \delta$.
 
 \vspace{0.3cm}
 In the case under inspection, when $M$ is compact, $\Imag d$ is closed because it is of finite codimension inside the closed subspace $\Ker d$. 
\end{proof}

We complete now Lefschetz duality by describing the other isomorphism. Notice first that we have two short exact sequences of chain complexes. The first one is:
\begin{equation}\label{exseqI} 0\ra{\Omega}^{*-1}(\partial M)\stackrel{a}{\longrightarrow} \Omega^*(M,\partial M)\stackrel{b}{\longrightarrow}  \widetilde\Omega^*(M)\ra 0
\end{equation}
\[ a(\gamma)=(0,\gamma)\qquad\qquad b(\omega,\gamma)=\omega.
\]
where $\widetilde{\Omega}^{*}( M)$ is $\Omega^{*}(M)$ with the differential changed to $-d$ in order for everything to commute.

The second one is:
\begin{equation}\label{exseqII} 0\ra \widetilde{\mathscr{D}}'_*(M)\stackrel{A}{\longrightarrow}  \mathscr{D}'_*(M,\partial M)\stackrel{B}{\longrightarrow}{ \mathscr{D}}'_{*-1}(\partial M)\ra 0
\end{equation}
\[ A(T)=(T,0)\qquad\qquad B(T,S)=S.
\]
where $\widetilde{ \mathscr{D}}'_{*}( M)$ is $\mathscr{D}'_*(M)$ with the differential changed to its negative.

\begin{lemma} The connecting homomorphism in the long exact sequence induced by (\ref{exseqI}), respectively (\ref{exseqII})  is the pull-back:
\[\iota^*: H^k_{\dR}(M)\ra H^{k}_{\dR}(\partial M).
\]
respectively the push-forward:
\[ \iota_*:H_{k}(\mathscr{D}'_*(\partial M))\ra H_k(\mathscr{D}'_*(M)).
\]
\end{lemma}
\begin{proof} Both relations are immediate consequences of the expression of the boundary homomorphism (see \cite{Br}, page 178) which, for example in the case of the first exact sequence looks like:
\[ \delta[\omega]=[a^{-1}\circ d\circ b^{-1}(\omega)].
\]
One can take $(\omega,0)$ for $b^{-1}(\omega)$ and then $d(\omega,0)=(0,\iota^*\omega)=a(\iota^*\omega)$.
\end{proof}
\begin{remark} One can also describe the connecting homomorphism \linebreak $H^{*-1}_{\dR}(\partial M)\ra H^*(\Omega_D^*(M, \partial M))$ which results by  composing $a$ with the inverse of the isomorphism (\ref{eq1234}). It takes $[\gamma]$ to $[d\alpha]$ where $\alpha$ is any  form on $M$ such that $\iota^*\alpha=\gamma$. Notice that $d\alpha$ is not necessarily exact in $\Omega_D^*(M)$ since $\alpha\notin\Omega_D^*(M)$, unless $\gamma=0$. This is the connecting homomorphism in the long exact sequence associated to:
\[ 0\ra \Omega^*_D(M)\ra\Omega^*(M)\stackrel{\iota^*}{\ra} \Omega^*(\partial M)\ra 0.
\]
\hfill\(\Box\)
\end{remark}
For the next result we will use the following notation: $H_k(\cdot):=H_k(\mathscr{D}'_*(\cdot))$,
 \[ \tau_{n,k}=k(n-k-1) \qquad \mbox\qquad\upsilon_{n,k}=k(n-k).\]
\begin{theorem} The following diagram with the horizontal rows exact commutes up to the indicated signs. Moreover the vertical arrows are isomorphisms:

\begin{equation}\label{commladder2} \xymatrix@C=1.5pc{ \ar[r]^{b\qquad\;}&H^k_{\dR}(M) \ar[r]^{\iota^*}\ar[d]_{\mathscr{L}_{II}} \ar@{}| {(-1)}^{\quad\tau_{n,k}}[dr] & H^k_{\dR}(\partial M) \ar[r]^{a\;\;\;\;} \ar[d]^{\PD_{\partial M}}  \ar@{}| {1}^{\;\; \;\;}[dr]& H^{k+1}_{\dR}(M,\partial M)\ar[r]^{\;\;\;b} \ar[d]_{\mathscr{L}_{I}} \ar@{}| {(-1)\;\;}^{\;\; \;\;\;\;\upsilon_{n,k+1}}[dr] & H^{k+1}_{\dR}(M) \ar[d]^{\mathscr{L}_{II}}\ar[r]^{\qquad\iota^*}&\\
\ar[r]^{A\qquad\quad}&H_{n-k}(M,\partial M)\ar[r]^{B} &H_{n-k-1}(\partial M) \ar[r]^{\iota_*} & H_{n-k-1}(M)\ar[r]^{A\quad} & H_{n-k-1}(M,\partial M)\ar[r]^{\qquad \;\;\;\;B}&
}
\end{equation}

where $\PD_{\partial M}$ is Poincar\'e isomorphism on the boundary.
\end{theorem}
\begin{proof} The signs are checked and explained via the skew-commutativity of the exterior product. Finally, $\mathscr{L}_{II}$ is  also an isomorphism because of the Five Lemma.
\end{proof}
\begin{remark} There is a bit of freedom in choosing the signs in the above definitions. Most of these choices are not independent. For example, one other choice would be to take the differential on $\Omega^k(M,\partial M)$ to be $(\omega,\gamma)\ra (d\omega,\iota^*\omega-d\gamma)$\footnote{This is the convention  that Bott and Tu use in \cite{BT}.} which requires a change of sign in $\mathscr{L}_I$ as the second integral will need a $-$ in front. Then the differential on $\mathscr{D}'_{k}(M,\partial M)$ has to be changed into $(T,S)\ra (\iota_*S+dT,-dS)$. This convention has the "drawback" of making the pair $(N,-\partial N)$  a closed relative current, meaning that we need to orient $\partial N$ using the \emph{inner} normal first in order for the pair $(N,\partial N)$ to be  closed. That, of course, seemed like heresy and therefore we made the choice (\ref{deq1}).   It turns out that once  (\ref{deq2}) is fixed then one does not have too much freedom in playing around with the  signs.
\end{remark}

\section{Lefschetz-de Rham Duality in bigger codimension}\label{LDbcodim}

In this section $K\subset M$ is an oriented, closed $k$-dimensional submanifold of $M$. Poincar\'e Duality on $M\setminus K$ says that the pairing:
\[ \Omega_{\cpt}^i(M\setminus K)\times \Omega^{n-i}(M\setminus K)\ra \bR,\qquad (\omega,\eta)\ra \int_{M\setminus K} \omega\wedge \eta,
\]
descends to a non-degenerate pairing in cohomology. This gives rise to two natural isomorphisms:
\begin{equation}\label{eq5.1} H_{\cpt}^{n-i}(M\setminus K)\simeq H_{i}(\mathscr{E}_*'(M\setminus K))
\end{equation}
and
\begin{equation}\label{eq5.2} H^{n-i}(M\setminus K)\simeq H_{i}(\mathscr{D}_*'(M\setminus K)).
\end{equation}
The groups $\mathscr{E}_*'(M\setminus K)$ are currents in $M\setminus K$ with compact support. 

We would like to produce isomorphisms which involve the relative cohomology/homology groups $H^*(M,K)$ and $H_*(M,K):=H_*(\mathscr{D}_*'(M,K))=H_*(\mathscr{E}_*'(M,K))$. 
We notice the following:
\begin{prop}\label{cptrel} The map
\begin{equation}\label{fiso}\Omega^i_{\cpt}(M\setminus K)\ra \Omega^i(M,K),\qquad \omega\ra(\omega,0),
\end{equation}
induces an isomorphism of cohomology groups where, on the right, $\omega$ is extended by $0$ outside its support.
\end{prop}
\begin{proof} Let $\Omega^*_D(M,K)$ be the space of forms on $M$ whose pull-back to $K$ vanishes. Then, just as in Lemma \ref{drel}  the map
\[ \Omega^*_D(M,K)\ra \Omega^*(M,K),\qquad \omega\ra (\omega,0)
\]
induces an isomorphism in cohomology. It is therefore enough to show the isomorphism in cohomology  for the extension by zero map
 \[\Omega_{\cpt}^*(M\setminus K)\ra \Omega_D^*(M,K).\]

We prove surjectivity first. Let $U_1\subset U_2 $ be two open tubular neighborhoods of $K$ such that $\overline{U}_1\subset U_2$ . There exists a smooth map  $\tilde{\pi}:M\ra M$   map with the following properties:
\begin{itemize}
\item[(i)] $\tilde{\pi}\bigr|_{U_1}$ projects $U_1$ radially to $K$, i.e. up to a diffeomorphism it is the projection of the normal bundle of $K$ to $K$;
\item[(ii)] $\tilde{\pi}\bigr|_{M\setminus U_2}=\id_{M\setminus U_2}$.
\item[(iii)] $\tilde{\pi}$ is homotopic to the identity $\id_M$.
\end{itemize}
Such a map can be constructed first on the normal bundle\footnote{Suppose a Riemannian metric has been fixed on $M$} $\nu K$ by letting 
\[ v\ra \rho(r)v,
\] 
where $\rho:[0,\infty)\ra [0,1]$ is a smooth function with $\rho(r)\equiv 1$ for $r\geq 2$ and $\rho(r)\equiv 0$ for $r\leq 1$. Then use a tubular neighborhood diffeomorphism  to move this map to $M$ and extend by identity.

Let $\omega\in \Omega_D^*(M,K)$ be a closed form.
Then $\tilde{\pi}^*\omega-\omega=d\beta$ for some $\beta$. Also, since  $\omega$ vanishes on $K$ then $\tilde{\pi}^*\omega$ vanishes on $U_1$. Hence $\tilde{\pi}^*\omega$ is the form we were looking for.

For injectivity, let $\omega\in \Omega^k_{\cpt}(M\setminus K)$ be  such that $\omega=d\eta$ with $\iota^*\eta=0$  on $K$.  Choose $K\subset U$, a tubular neighborhood, such that $\omega\bigr|_{U}\equiv 0$. The inclusion map $K\hookrightarrow U$   is a homotopy equivalence and therefore $d\eta\bigr|_U=0$ and $\iota^*\eta=0$  imply that there exists a form $\alpha\in\Omega^{k-2}(U)$ such that $d\alpha=\eta$ on $U$. Now let $\phi:U\ra \bR$ be a function with compact support such that $\phi\equiv 1$ in a neighborhood of $K$. Then $\tilde{\eta}:=\eta -d(\phi \alpha)$ is a well-defined form on $M$ such that $d\tilde{\eta}=\omega$ and $\supp \tilde{\eta}\subset M\setminus K$.

\end{proof}
 An immediate consequence is the next one: \\
\noindent
\emph{Alternative proof to Theorem \ref{LefDu}}: Use Proposition \ref{cptrel} to reduce the proof to the fact that the pairing:
\[ H^i_{\cpt}(M\setminus \partial M)\times H^{n-i}(M),\qquad (\omega,\eta)\ra \int_{M}\omega\wedge \eta
\]
is non-degenerate. But this is Poincar\'e Duality on the non-compact manifold $M\setminus \partial M$.
\hfill\(\Box\)
\vspace{0.3cm}

The  dual result of Proposition \ref{cptrel}  is the following statement for currents:
\begin{prop} \label{currel} The map
\begin{equation}\label{eq5.3} \mathscr{D}_*'(M,K)\ra \mathscr{D}_*'(M\setminus K),\qquad (T,S)\ra T\bigr|_{M\setminus K},
\end{equation}
induces an isomorphism of homology groups.
\end{prop}
\begin{proof} Let  $ U\supset K$ be an open tubular neighborhood. We will use again a smooth map $\tilde{\pi}:M\ra M$ which radially collapses $\overline {U}$ to $K$ and is homotopic to the identity. We orient $\partial U$ as the boundary of $U^c$ and write it $\partial U^c$ to recall this orientation convention. 

Consider the map:
\begin{equation}\label{inv5.3}\varphi:\Omega^i(M\setminus K)\ra \mathscr{D}_{n-i}'(M,K),\end{equation}\[ \varphi(\omega)= \left(\eta\ra \int_{{U}^c}\tilde{\pi}^*\eta\bigr|_{U^c}\wedge\omega\bigr|_{U^c}  ,\; \gamma\ra \int_K  \left(\gamma\wedge \int_{\partial U^c/K}j^*\omega\right) \right),
\]
where $j:\partial U^c\ra M\setminus K$ is the inclusion map. Notice that:
\[ \varphi(\omega)(\eta,\gamma)=\mathscr{L}_{II}^{U^c}(\omega\bigr|_{U^c})\left(\left(\tilde\pi\bigr|_{U^c}\right)^*\eta,\left(\tilde\pi\bigr|_{\partial U^c}\right)^*\gamma\right)\;\mbox{or}
\] 
\[ \varphi=\left(\tilde{\pi}\bigr|_{U^c}\right)_*\circ \mathscr{L}_{II}^{U^c}.
\]
Then by (\ref{eq3}), (\ref{eq4}) and Remark \ref{pd} we have:
\[d[\varphi(\omega)]=(-1)^{n-1-|\omega|}\varphi(d\omega),
\]
and we get a well-defined map 
\begin{equation}\label{eq5.3h} \varphi:H^i(M\setminus K)\ra H_{n-i}(M,K).
\end{equation}
We check that composing  (\ref{eq5.3h}) with the induced map in homology for (\ref{eq5.3}) gives the Poincar\'e Duality pairing
 \begin{equation}\label{eq5.31} H^i(M\setminus K)\ra H_{n-i}(\mathscr{D}_*(M\setminus K)),\qquad
\omega\ra \left(\eta\ra \int_{M\setminus K}\eta\wedge \omega\right).
\end{equation}
Indeed, notice that for $\eta\in \Omega^*(M)$ with $\supp\eta\subset M\setminus K$ :
\[ \int_{U}\tilde{\pi}^*\eta\wedge\omega=0,
\]
Hence for a closed $\omega$ and closed $\eta$ with $\supp\eta\subset M\setminus K$ we have:
\[ \int_{U^c}\tilde{\pi}^*\eta\wedge\omega=\int_M\tilde\pi^*\eta\wedge\omega=\int_M\eta\wedge \omega=\int_{M\setminus K}\eta\wedge \omega,
\]
where in the last equality we used that $\tilde{\pi}$ is homotopic to the identity and therefore $\eta$ and $\tilde{\pi}^*\eta$ differ by an exact form.
This proves surjectivity of   (\ref{eq5.3}) on homology groups.

 The following commutative ladder, together with the $5$-Lemma proves that in fact (\ref{eq5.3h}) is an isomorphism and therefore so is (\ref{eq5.3}):
\begin{equation}\label{diag2}\footnote{Recall $\upsilon_{n,k}=k(n-k)$.}\xymatrix@C=2.2pc{  H^{k-i}(K)\ar[r]^{\tau \wedge \tilde{\pi}^*(\cdot)}\ar[d]_{\PD_K} \ar@{}| {1}[dr]&H^{n-i}(M)\ar[r]^{(\cdot)\large|_{M\setminus K}\;\;\;}\ar[d]_{\PD_M} \ar@{}| {(-1)}^{\quad\upsilon_{n,i}}[dr]   &    H^{n-i}(M\setminus K) \ar[r]^{\beta}\ar[d]_{\varphi} \ar@{}| {(-1)\;\;\;\;}^{\quad\upsilon_{k,i-1}-1}[dr] & H^{k-i+1}(K)\ar[r]^{\tau\wedge \tilde{\pi}^*(\cdot)}\ar[d]^{\PD_K} \ar@{}| {1}[dr]& H^{n-i+1}(M)\ar[d]^{\PD_M}
\\ H_{i}(K)\ar[r] ^{\iota_*}&
 H_i(M)\ar[r]^{A}& H_i(M,K)\ar[r]^B & H_{i-1}(K)\ar[r]^{\iota_*}& H_{i-1}(M) }
\end{equation}
where the bottom horizontal maps come from the analogue of (\ref{exseqII}); $\tau\in\Omega^{n-k}(M)$ is a Thom form for $K$   with \emph{compact} support in $U$,  $\PD_M,$ and $\PD_K$ are the Poincar\'e Duality operators: 
\[\omega\ra\left \{\eta \ra \int_{M}\omega\wedge\eta\right\}.\]
Finally,
\[ \beta (\gamma)=-\int_{\partial U^c/K} j^*\gamma=\int_{\partial U/K}j^*\gamma.\]

The top exact sequence is really the exact sequence of the pair $(M, M\setminus K)$ combined with  excision:
\begin{equation}\label{eq5.4} H^{n-i+1}(M,M\setminus K)\stackrel{\sim}{\ra} H^{n-i+1}(\overline{U},\overline{U}\setminus K),
\end{equation}
and with the Thom isomorphism:
\begin{equation}\label{eq5.5} H^{n-i+1}(\overline{U},\overline{U}\setminus K)\stackrel{\sim}{\ra} H^{k-i+1}(K),
\end{equation}
all made explicit as follows. The map $(\ref{eq5.4})$ is  pull-back of pairs.  The isomorphism  (\ref{eq5.5}) is:
\[ (\omega,\gamma)\ra \int_{\overline{U}/K}\omega-\int_{\partial U^c/K}j^*\gamma
\]
which results as a composition of the pull-back  $H^{n-i+1}(\overline{U},\overline{U}\setminus K)\ra H^{n-i+1}(\overline{U},\partial U^c)$ and the map $\nu:H^{n-i+1}(\overline{U},\partial U)\ra H^{k-i+1}(K)$ defined at (\ref{eq6.3}). Therefore we get the isomorphism
 \[H^{n-i+1}(M,M\setminus K)\ra H^{k-i+1}(K),\;\;\;\;  \tilde{\nu}(\omega,\gamma)=\int_{\overline{U}/K}\omega\bigr|_{U}-\int_{\partial U^c/K}j^*\gamma\bigr|_{U\setminus K},
\]
which composed with $\gamma\ra(0,\gamma)$ gives exactly the map $\beta$. The commutativity of the third square up to the indicated sign is easy to verify.
 
Let us check the commutativity of the last (and the first) square. 
 
  Let $\gamma\in \Omega^{k-i+1}(K)$, $\omega\in \Omega^{n-i+1}(M)$ be two closed forms. Then
\[ \PD_M(\tau\wedge\tilde{\pi}^*(\gamma))(\omega)=\int_M\tau\wedge \tilde{\pi}^*\gamma\wedge\omega=\int_{\overline{U}}\tau\wedge \tilde{\pi}^*\gamma\wedge\omega\stackrel{\#}{=}\int_{\overline{U}}\tau\wedge\tilde{\pi}^*\gamma\wedge \tilde{\pi}^*\iota^*\omega=\]\[=\int_K\left(\int_{\overline{U}/K}\tau\right)\wedge \gamma\wedge \iota^*\omega=\int_K\gamma\wedge \omega=\iota_*(\PD_K(\gamma))(\omega).\]
Equality in $\#$ follows the fact that $\iota \circ \tilde{\pi}:\overline{U}\ra \overline{U}$ is  homotopy equivalent with the identity and hence $\omega$ and $ \tilde{\pi}^*\iota^*\omega$  differ by an exact form $d\theta$. Combine this with Stokes on $U$ for the form $d(\tau\wedge\tilde{\pi}^*\gamma\wedge \theta)$ to get $\#$.

We show now that the natural map 
\[ H^{n-i+1}(M,M\setminus K)\ra H^{n-i+1}(M),\qquad (\omega,\gamma)\ra \omega\] becomes
 \[H^{k-i+1}(K)\ni\gamma\ra \tau\wedge \tilde{\pi}^*(\gamma)\in H^{n-i+1}(M) \]
 after the identification of $H^{n-i+1}(M,M\setminus K)$ with $H^{k-i+1}(K)$ via the Thom isomorphism as above. This means
 \[ \int_M\tau\wedge \tilde{\pi}^*(\tilde\nu(\omega,\gamma))\wedge\eta=\int_{M} \omega\wedge\eta,
 \]
 for all closed $\eta\in\Omega^{i-1}(M)$ and closed pairs $(\omega,\gamma)$.
 In view of the commutativity of the last square this is the same thing as:
 \begin{equation}\label{eq5.7}\int_K\tilde\nu(\omega,\gamma)\wedge\iota^*\eta=\int_{M} \omega\wedge\eta.
 \end{equation}
 The left hand side of (\ref{eq5.7}) is:
 \[ \int_{\overline{U}} \omega\bigr|_{U}\wedge\tilde{\pi}^*\iota^*\eta-\int_{\partial U^c}\gamma\bigr|_{\partial U^c}\wedge\tilde{\pi}^*\iota^*\eta=\int_{\overline{U}}\omega\wedge\tilde{\pi}^*\iota^*\eta-\int_{U^c}d(\gamma\wedge\tilde{\pi}^*\iota^*\eta)=
 \]\[=\int_{\overline{U}}\omega\wedge\tilde{\pi}^*\iota^*\eta+\int_{U^c}\omega\wedge\tilde{\pi}^*\iota^*\eta=\int_M\omega\wedge \tilde{\pi}^*\iota^*\eta=\int_M\omega\wedge\eta.
 \]
 
Finally, we prove the commutativity of the second square of diagram (\ref{diag2}) up to the indicated sign. We have:
\begin{equation}\label{eq59} (-1)^{\upsilon_{n,i}}A\circ \PD_M(\omega)=\left(\eta\ra \int_{M}\eta\wedge \omega,\gamma\ra 0\right)
\end{equation}
Since $\tilde{\pi}\sim \id$ then
\begin{equation}\label{eq529} \eta-\tilde{\pi}^*\eta=d\tilde{\eta}+\tilde{H}(d\eta),
\end{equation}
for some form $\tilde{\eta}$ where $\tilde{H}:\Omega^{i+1}(M)\ra\Omega^i(M)$ is an operator induced by the homotopy between $\tilde{\pi}$ and $\id$ (see (\ref{eq6.17})). Then since $\omega$ is closed:
\[ \int_{M}\eta\wedge\omega=\int_{M}\tilde{\pi}^*\eta\wedge \omega+\int_M\tilde{H}(d\eta)\wedge\omega.
\]
Now
\begin{equation}\label{eq5.20} \left(\int_{U^c}\tilde{\pi}^*\eta\wedge\omega,\int_{\partial U^c}\tilde{\pi}^*\gamma\wedge\omega\right)-\left(\int_{M}\tilde{\pi}^*\eta\wedge \omega,0\right)=\left(-\int_{U}\tilde{\pi}^*\eta\wedge \omega,\int_{\partial U^c}\tilde{\pi}^*\gamma\wedge\omega\right)=
\end{equation}
Let $\omega=\tilde{\pi}^*\omega+d\theta$. Then (\ref{eq5.20}) becomes:
\begin{equation}\label{eq19}\left(-\int_{U}\tilde{\pi}^*\eta\wedge d\theta,\int_{\partial U^c}\tilde{\pi}^*\gamma\wedge d\theta\right)=\end{equation}\[=(-1)^{|\eta|-1}\left(\int_U[d(\tilde{\pi}^*\eta\wedge\theta)-\tilde{\pi}^*d\eta\wedge\theta],\int_{\partial U^c}-d\tilde{\pi}^*\gamma\wedge\theta\right)=\]\[=(-1)^{|\eta|}\left(\int_{\partial U^c}\tilde{\pi}^*\eta\wedge\theta+\int_{U}\tilde{\pi}^*d\eta\wedge\theta,\int_{\partial U^c}\tilde{\pi}^*d\gamma\wedge\theta\right).
\]
If we let
 \[S(\hat{\gamma}):=(-1)^{|\hat{\gamma}|}\int_{\partial U^c}\tilde{\pi}^*\hat{\gamma}\wedge \theta \mbox{~and~}\]\[ T(\hat{\eta}):=\int_{M}\tilde{H}(\hat{\eta})\wedge\omega+(-1)^{|\hat{\eta}|}\int_{U}\tilde{\pi}^*\hat{\eta}\wedge\theta,\]
then putting together (\ref{eq59}), (\ref{eq529}) and (\ref{eq19}) one gets:
\[[\varphi(\omega\bigr|_{M\setminus K})-(-1)^{\upsilon_{n,i}}A\circ \PD_M(\omega)](\eta,\gamma)=(S(\eta)-dT(\eta),dS(\gamma))=d(T,S)(\eta,\gamma).
\]
 \end{proof}
\begin{remark}
  The inverse  in homology  (\ref{inv5.3}) of (\ref{eq5.3}) was initially obtained  via the following sequence of  isomorphisms:
\[ H_i(\mathscr{D}_{*}'(M,K))\;\simeq\; H_i(\mathscr{D}_*'(M,\overline{U}))\;\simeq\;H_i(\mathscr{D}_*'(M\setminus U,\overline{U}\setminus U))\;\simeq\; H^{n-i}(M\setminus U)\;\simeq\;  \]\[{\simeq}\; H^{n-i}(M\setminus \overline{U});\simeq\; (H^i_{\cpt}(M\setminus \overline{U}))^*\;\simeq\; H_i(\mathscr{D}_*'(M\setminus \overline{U}))\;\simeq \;H_i(\mathscr{D}_*'(M\setminus K)).
\]

This sequence of isomorphisms  would constitute an alternative proof provided one first proves the excision property for $H_*(\mathscr{D}_*'(\cdot,\cdot))$ used in the third isomorphism above. Since the later are really the Borel-Moore relative homology groups this is a consequence of the general theory (see \cite{BM}, Section 5).  However, we wanted to stay away from sheaf theory and keep everything as elementary as possible. Notice that a posteriori, excision is also a consequence of Proposition \ref{currel}. 
\end{remark}
\begin{remark} It was long known via sheaf theory that $H_*(\mathscr{D}_*'(M\setminus K))$ fits into a long exact sequence:
\[ H_*(\mathscr{D}_*'(K))\ra H_*(\mathscr{D}_*'(M))\ra H_*(\mathscr{D}_*'(M\setminus K))\ra H_{*-1}(\mathscr{D}_{*}'(K)).
\]
The isomorphism (\ref{eq5.3}) seems to give a direct explanation to this sequence, at least when $K$ is a compact oriented submanifold of $M$.
\end{remark}

Putting together Propositions \ref{cptrel} and \ref{currel} with  (\ref{eq5.1}) and (\ref{eq5.2}) we get:
\begin{theorem}\label{thldbc} There exist Lefschetz Duality isomorphisms:
\begin{itemize}
\item[(a)]
\[ H^{n-i}(M,K)\simeq H_{i}(\mathscr{E}_*'(M\setminus K))=:H_{i}(M\setminus K).
\]
\item[(b)]
\[ H^{n-i}(M\setminus K)\simeq H_{i}(\mathscr{D}_*'(M,K))=:H_{i}(M,K).
\]
\end{itemize}
\end{theorem}
\vspace{0.3cm}
The isomorphisms in Theorem \ref{thldbc} are not induced \emph{directly} by bilinear pairings. To get indeed pairings, one  has to invert (\ref{fiso}) and (\ref{eq5.3}). While the inverse to (\ref{fiso}) has a rather straightforward description which  can find in the proof,  in the case of (\ref{eq5.3}) the inverse is given by the map (\ref{inv5.3}), provided one can represent a closed current in $M\setminus K$ by a closed form.

\begin{example} One quick application of Theorem \ref{thldbc} is the (homological) definition of the intersection number of two oriented manifolds of complementary dimension, one of them compact without boundary $L\subset M\setminus K$ and the other one with boundary $(S,\partial S)\subset (M,K)$. This is because $L\in H_*(\mathscr{E}'_*(M\setminus K))$ while $(S,\partial S)\in H_*(\mathscr{D}_*'(M,K))$.  This is one way the linking number of a pair of knots $(L,K)$ can be defined, by taking $S$ to be a Seifert surface for $K$.
\end{example}

\section{Homotopy operators on manifolds with boundary}

We make the following important observation: the boundary of $\partial B$ is oriented using the outer normal first convention. Now, the product orientation $\partial_t\wedge \Or{B}$ of $[0,t]\times \partial B$ is the opposite  orientation of $[0,t]\times \partial B$ as a codimension $1$ boundary of $[0,t]\times B$, using the same outer normal first convention. \emph{We will always use the product orientation on $[0,t]\times \partial B$}. Hence Stokes Theorem on the cylinder $[0,t]\times B$ takes the form:
\begin{equation}\label{Steq1}\int_{[0,t]\times B} d\alpha=\int_{B}\iota_t^*\alpha-\int_B\iota_0^*\alpha-\int_{[0,t]\times \partial B}\alpha.
\end{equation}

We consider the following situation. Suppose $(B,\partial B)$ and $(P,\partial P)$ are two \emph{compact} smooth manifolds with boundary with $\dim{B}=n$ and $\phi:[0,t]\times B\ra P$ is a homotopy which is boundary compatible in the sense that 
\[ \Imag\phi\bigr|_{[0,t]\times \partial B}\subset \partial P.
\]
Consider the following homotopy operators:
\[ \eT^{\mathrm{I}}_t:\Omega^k(P,\partial P)\ra \mathscr{D}'_{n-k+1}(B),\]\[ \eT^{\mathrm{I}}_t(\omega,\gamma)(\eta)=-\int_{[0,t]\times B}\phi^*\omega\wedge\pi_2^* \eta+\int_{[0,t]\times\partial B}\phi^*\gamma\wedge\pi_2^*\eta,
\]
\vspace{0.4cm}
\[\eT^{\mathrm{II}}_t:\Omega^k(P)\ra \mathscr{D}'_{n-k+1}(B,\partial B),\]\[\eT^{\mathrm{II}}_t(\eta)(\omega,\gamma)=\left(-\int_{[0,t]\times B}\pi_2^*\omega\wedge \phi^*\eta,\int_{[0,t]\times \partial B}\pi_2^*\gamma\wedge \phi^*\eta\right)
\]
where $\pi_2:[0,t]\times B\ra B$ is the obvious projection. Recalling (\ref{deq1}) and (\ref{deq2}) we have:

\begin{prop}\label{heq}\begin{itemize}
\item[] \begin{equation}\label{heq1} \eT^{\mathrm{I}}_t(d(\omega,\gamma))+(-1)^{k-1}d\eT^{\mathrm{I}}_t(\omega,\gamma)=\mathscr{L}_I(\phi_t^*\omega,\phi_t^*\gamma)-\mathscr{L}_I(\phi_0^*\omega,\phi_0^*\gamma).
\end{equation}
\item[]
\vspace{-0.5cm}
\begin{equation}\label{heq2} (-1)^{n-k} \eT^{\mathrm{II}}_t(d\eta)+d\eT^{\mathrm{II}}_t(\eta)=\mathscr{L}_{II}(\phi_t^*\eta)-\mathscr{L}_{II}(\phi_0^*\eta).
\end{equation}
\end{itemize}
\end{prop}
\begin{proof} For (\ref{heq1}):
\[\eT^{\mathrm{I}}_t(d(\omega,\gamma))(\eta)=\eT^{\mathrm{I}}_t(-d\omega,\iota^*\omega+d\gamma)(\eta)=\]\[=
\left(\int_{[0,t]\times B}d\phi^*\omega\wedge\pi_2^*\eta+\int_{[0,t]\times \partial B}\phi^*\iota^*\omega\wedge \pi_2^*\eta\right)+\int_{[0,t]\times\partial B}d\phi^*\gamma\wedge\pi_2^*\eta.
\]
\[(-1)^{k-1}\eT^{\mathrm{I}}_t(\omega,\gamma)(d\eta)=(-1)^k\int_{[0,t]\times B}\phi^*\omega\wedge d\pi_2^*\eta+(-1)^{k-1}\int_{[0,t]\times \partial B}\phi^*\gamma\wedge d\pi_2^*\eta.
\]
Summing up and using Stokes (\ref{Steq1}) one gets:
\[\eT^{\mathrm{I}}_t(d(\omega,\gamma))(\eta)+(-1)^{k-1}\eT^{\mathrm{I}}_t(\omega,\gamma)(d\eta)=\int_B\phi_t^*\omega\wedge \eta-\int_{B}\phi_0^*\omega\wedge\eta+\int_{\partial B}\phi_t^*\gamma\wedge\eta-\int_{\partial B}\phi_0^*\gamma\wedge\eta
\]
For (\ref{heq2}):
 \[ (-1)^{n-k}\eT^{\mathrm{II}}_t(d\eta)(\omega,\gamma)=(-1)^{n-k}\left( -\int_{[0,t]\times B}\pi_2^* \omega\wedge d\phi^*\eta, \int_{[0,t]\times \partial B} \pi_2^*\gamma\wedge d\phi^*\eta \right).
\]
\[  \left(d\eT^{\mathrm{II}}_t(\eta)\right)(\omega,\gamma)=\left(\int_{[0,t]\times \partial B}\pi_2^*\iota^*\omega\wedge \phi^*\eta+\int_{[0,t]\times B}d\pi_2^*\omega\wedge \phi^*\eta,\int_{[0,t]\times \partial B} d\pi_2^*\gamma\wedge\phi^*\eta\right)
\]
Summing up and using Stokes again one gets:
\[(-1)^{n-k}\eT^{\mathrm{II}}_t(d\eta)(\omega,\gamma)+\left(d\eT^{\mathrm{II}}_t(\eta)\right)(\omega,\gamma)=\left(\int_B\omega\wedge(\phi_t^*\eta-\phi_0^*\eta),\int_{\partial B}\gamma\wedge(\phi_t^*\eta-\phi_0^*\eta)\right)
\]
\end{proof}

\begin{corollary}
\begin{itemize}
\item[(a)] If $\phi:[0,t]\times B\ra P$ is a boundary compatible homotopy and $(\omega,\gamma)\in \Omega^k(P,\partial P)$ is a closed pair then $\mathscr{L}_I(\phi_t^*\omega,\phi_t^*\gamma)$ and $\mathscr{L}_I(\phi_0^*\omega,\phi_0^*\gamma)$ represent the same homology class in $B$.
\item[(b)] If $\phi:[0,t]\times B\ra P$ is a boundary compatible homotopy and $\eta\in \Omega^k(P,\partial P)$ is a closed form then $\mathscr{L}_{II}(\phi_t^*\eta)$ and $\mathscr{L}_{II}(\phi_0^*\eta)$ represent the same relative homology class in $B$.
\end{itemize}
\end{corollary}
In the next section we will analyze a situation where  $\mathscr{L}_{II}(\phi_t^*\eta)$ and $\mathscr{L}_I(\phi_t^*\omega,\phi_t^*\gamma)$ degenerate as $t\ra \infty$ to rectifiable currents.

\section{Chern-Gauss-Bonnet on manifolds with boundary}

 If $M$ is an oriented, Riemannian manifold with boundary of dimension $n$ with $n$ even and $\nabla$ is the Levi-Civita connection then Chern proved in \cite{Ch2} the following:
\begin{equation}\label{eq5.01} \chi(M)=\int_M\Pf(\nabla)-\int_{\partial M}\TPf(\nabla).
\end{equation}
Here $\chi(M)$ is the Euler characteristic of the manifold, $\Pf(\nabla)$ is the Pfaffian form associated to the curvature $F(\nabla)$, while $\TPf(\nabla)$ is a transgression form. Recall that transgression forms for an invariant polynomial $P$ arise in general by "comparing" the closed forms associated to  two different connections via  $P$. The construction in the case of the Pfaffian proceeds as follows. Let $\nabla^1,\nabla^2$ be two \emph{metric compatible} connections on a vector bundle $E\ra B$ and let 
\begin{equation}\label{tina} \tina=\frac{d}{dt}+(1-t)\nabla^1+t\nabla^2
\end{equation}
be a connection on $\pi_2^*E\ra B$ where $\pi_2:[0,1]\times B\ra B$ is the projection. The operator $\frac{d}{dt}$ is a partial connection:
\[ \frac{d}{dt}s:=\frac{\partial s}{\partial t}dt
\]
and extends via the usual Leibniz rule to act on forms with values in $E$. Then
\begin{equation}\label{transdef} \TPf(\nabla^1,\nabla^2):=\int_{[0,1]}\Pf(\tilde{\nabla}),
\end{equation}
where $\int_{[0,1]}$ represents integration over the fiber of $\pi_2$. Then
\begin{equation}\label{hfor} \Pf(\nabla^2)-\Pf(\nabla^1)=d\TPf(\nabla_1,\nabla_2),
\end{equation}
which follows from  the homotopy formula that we recall. Let $H:[0,1]\times N\ra M$ be a smooth homotopy between smooth manifolds and assume that $N$ compact. Then:
\begin{equation}\label{eq6.17} \int_{[0,1]}H^*d\omega+d\left(\int_{[0,1]}H^*\omega\right)=H_1^*\omega-H_0^*\omega,\qquad \forall \omega\in \Omega^*(M).
\end{equation}
In order to get (\ref{hfor}), one applies this to $\id:I\times B\ra I\times B$ and $\omega=\Pf(\tilde{\nabla})$  taking into account that
\[ F(\tilde{\nabla})=dt\wedge(\nabla^2-\nabla^1)+F(\nabla_t),\qquad \nabla_t=(1-t)\nabla^1+t\nabla^2.
\]
\[ \iota_t^*(F(\tilde{\nabla}))=F(\nabla_t),\;\;\qquad\forall \iota_t \; \mbox{slice inclusion}.
\]

 In the context of Chern's theorem one considers $\nabla^2=\nabla$ to be the Levi-Civita connection on $M$. Use the normal exponential    in order to identify a collar neighborhood $\partial M\subset U\simeq \partial M\times [0,1)$ of $\partial M$. Put  the  product metric on the collar, and consider $\nabla^{\nu}$ to be the Levi-Civita connection of the product metric . Notice that the unit normal vector becomes parallel and  consequently:
\[ \Pf({\nabla}^{\nu})=0.
\]
Along the collar neighborhood $U$ we get according to (\ref{hfor})
\[ \Pf(\nabla)=d\TPf(\nabla^{\nu},\nabla).
\]
The form $\TPf(\nabla):=\TPf(\nabla^{\nu},\nabla)$ is what appears in (\ref{eq5.01}).

Chern Theorem on a boundaryless manifold is  a combination of a Poincar\'e Duality statement which says that the Poincar\'e dual of the Pfaffian is the $0$-dimensional curent obtained by counting the isolated zeros of a vector field $X:M\ra TM$ in generic (transversal conditions) and Poincar\'e-Hopf Theorem  which insures that:
\[ \chi(M)=\#([X^{-1}(0)]),
\]
where $\#$ is the (signed) counting function. In the case of a manifold with boundary the situation is no different. However, when generalizing this to oriented, Riemannian vector bundles $E\ra B$ of even rank one has to pay attention to one point. The construction of the boundary integrand  in (\ref{eq5.01}) depends in an essential way on the fact that there exists a natural non-vanishing section of $TM\bigr|_{\partial M}$.

 For a general oriented vector bundle $E\ra B$  there exists a topological obstruction for that to happen, namely $e(E\bigr|_{\partial B})\in H^k(\partial B;\bR)$  might  be non-zero. In fact, that is the only obstruction. We proceed now to proving such a general version.

\begin{definition}\label{transdef} Let $B$ and $P$ be two manifolds with boundary an let $s:(B,\partial B)\ra (P,\partial P)$ be a smooth map of pairs, i.e. $s(\partial B)\subset \partial P$. Let $(N,\partial N)\subset(P,\partial P)$ be a submanifold with boundary. Then $s$ is  transversal to  $N$ if both  $s\bigr|_{\inte{M}}$ and $s\bigr|_{\partial M}$ are transverse to $N$ respectively $\partial N$.
\end{definition}
It is not hard to show that in the conditions of  Definition \ref{transdef}, 
\[\left(s^{-1}(N),\left(s\bigr|_{\partial B}\right)^{-1}(\partial N)\right)\subset (B,\partial B)\]
 is a submanifold with boundary.

 When $\pi:E\ra B$ is a Riemannian vector bundle of  rank $2k$ endowed with a connection $\nabla$ and $s:\partial B\ra E$ is a non-vanishing section then consider on $E\bigr|_{\partial B}$ the connection $\nabla^1$ defined as follows. Over the line bundle $\langle s\rangle$ trivialized via the normalized section $\frac{s}{|s|}$
 put the trivial connection. Over $\langle s\rangle^{\perp}$ consider the orthogonal projection $\iota^*\nabla^{\perp}$ of $\nabla$. Then
 \[ \nabla^1=d\oplus \nabla^{\perp}.
 \] 
 
   Now just as in (\ref{hfor}) there exists a form $\TPf(\nabla,s)\in\Omega^{2k-1}(\partial B)$  such that:
   \[ d\TPf(\nabla,s)=\Pf(\iota^*\nabla)-\Pf(\nabla^1)=\iota^*\Pf(\nabla).
 \]
 \begin{remark} We emphasize that it is important to trivialize the line bundle $\langle s\rangle$ via the \emph{normalized} section $\frac{s}{|s|}$ if $\nabla^1$ is to be metric compatible.
 \end{remark}
 
 The main result of this section is:

\begin{theorem}\label{CGB} Let $E\ra B$ be an oriented Riemannian vector bundle of even rank $2k$ endowed with a metric compatible connection $\nabla$ over the compact manifold with boundary $B$. Let $s:B\ra E$ a section which is transversal to the zero section of $E$.
\begin{itemize} 
\item[(a)]  Then $s^{-1}(0)$ and $\mathscr{L}_{II}(\Pf(\nabla))$ determine the same class in $H_{n-2k}(\mathscr{D}'_*(B,\partial B))$ .
\item[(b)] If  $s\bigr|_{\partial B}$ does not vanish, then  there exists a form $\TPf(\nabla,s)$ such that \linebreak $\mathscr{L}_I(\Pf(\nabla),-\TPf(\nabla,s))$ and $s^{-1}(0)$ represent the same class in $ H_{n-2k}(\mathscr{D}'_*(B))$.
\end{itemize}
\end{theorem}

The set-up for the proof is as follows.  Let $P:=S(\bR\oplus E)\ra B$ be the even dimensional spherical bundle associated to $\bR\oplus E$. Notice that $P$ is a compact manifold with boundary $\partial P$ such that the projection $\pi:P\ra B$ restricts to a fiber bundle map: $\pi:\partial P\ra \partial B$.

 Clearly $E$ embeds into $P$ via the inverse to the stereographic projection which we take to be:
\[\eSt: E\hookrightarrow S(\bR\oplus E),\qquad \eSt(v)=\frac{1}{1+|v|^2}(1-|v|^2,2v),\quad\quad v\in E.
\]
Notice that $0\in E$ corresponds to $(1,0)\in \bR\oplus E$. 

Denote by $\varphi_0:B\ra P$ the section $s$ with target $P$.

Now consider the vertical gradient flow  of the (negative height) function:
\[f:S(\bR\oplus E)\ra \bR,\qquad f(t,v)=-t
\]
Denote by $[0]$ and $[\infty]$ the two critical manifolds of  $X:=\nabla^V f$, the vertical gradient of $f$, which correspond to the north pole, resp. south pole sections of $P\ra B$. Notice that $P\ra B$ comes with two canonical smooth sections, namely the north pole/south pole sections. We sometimes  call them the zero/infinity sections for the obvious reasons.

  The field $X$ is undoubtedly the simplest example of a vertical, tame, horizontally constant, Morse-Bott-Smale vector field, using the terminology of \cite{Ci}. The flow it generates $\Theta:\bR\times P\ra P$ is boundary compatible, i.e. $\Theta(\bR\times \partial P)\subset \partial P$.  As a consequence, the stable manifolds:
\[ S([0])=[0]; \qquad S([\infty])=P\setminus [0],
\]
and the unstable manifolds
\[ U([0])=P\setminus [\infty];\qquad U([\infty])=[\infty].
\]
are all manifolds with boundary contained in $\partial P$.

Let $\varphi_t:B\ra P$ be the $1$-parameter family of sections induced by the flow  of $X$:
\[ \varphi_t(b):=\Theta_t(\varphi_0(b)).
\]

 If $\omega$ is a smooth form of degree $k$ on $B$ and $N$ is an oriented compact submanifold of dimension $p\geq k$ then
\[ \omega\wedge [N](\eta):=\int_N\omega\wedge \eta,\qquad \forall \eta \in\Omega^{p-k}(B),
\]
while
\[ [N]\wedge \omega(\eta):=\int_N \eta\wedge \omega,\qquad \forall \eta \in\Omega^{p-k}(B).
\]
The next Theorem parallels the general results obtained in \cite{Ci} in the case of the particular vertical flow induced by $X$. 
\begin{theorem} \label {transfor}Let $\eta\in \Omega^{2k}(P)$ be a closed form of degree equal to the dimension of the fiber of $P\ra B$ and let $(\omega,\gamma)\in \Omega^{2k}(P,\partial P)$ be a closed pair. Suppose $\varphi_0$ is transversal to $[0]$. Then
\begin{itemize}
\item[(a)]
\begin{equation}\label{flim}\lim_{t\ra \infty}\mathscr{L}_{I}(\varphi_t^*\omega,\varphi_t^*\gamma)=\left(\int_{P/B}\omega\right)\cdot [\varphi^{-1}_0[0]]+\widehat{\omega_{[\infty]}}\wedge [B]+\iota_*(\widehat{\gamma_{[\infty]}}\wedge [\partial B]),
\end{equation}
where $\widehat{\omega_{[\infty]}}$ and $\widehat{\gamma_{\infty}}$ are the restrictions of $\omega$ and $\gamma$  to $[\infty]$, pulled back via the infinity section to $B$ and $\partial B$ respectively.
\item[(b)]
\begin{equation}\label{seclim} \lim_{t\ra \infty}\mathscr{L}_{II}(\varphi_t^*\eta)=\left(\left(\int_{P/B} \eta\right)\cdot [\varphi_0^{-1}[0]]+[B]\wedge \widehat{\eta_{[\infty]}},\left(\int_{\partial P/\partial B} \eta\right)\cdot [\varphi_{0,\partial B}^{-1}[0]]+[\partial B]\wedge \widehat{\eta_{[\infty]}}\right)
\end{equation}
\end{itemize}
\end{theorem}
\begin{remark} All fiber integrals which appear in the statement of Theorem \ref {transfor} are in fact constant functions, since all forms to be integrated are closed. 
\end{remark}
\begin{proof} We will use the same geometric idea for both parts. We blow up $\varphi_0^{-1}[0]$ in $B$ and $[0]\sqcup [\infty]$ in $P$. 

Since $B_0:=\varphi_0^{-1}[0]$ is a submanifold with boundary in $(B,\partial B)$ the result of the oriented blow-up is a manifold with corners, having only corner-strata in codimension $1$ and $2$. We denote it by $\hat{B}$.    There exists a projection map
\[ \rho:\hat{B}\ra B
\]
which is a diffeomorphism on $\hat{B}\setminus \eE$ and is the projection of the spherical normal bundle $S(\nu B_0)\ra B_0$. We denoted by $\eE:=\rho^{-1}(B_0)\subset \hat{B}$ the exceptional locus, diffeomorphic to $S(\nu B_0)$. We denote by $\widehat{\partial B}\subset \hat{B}$ the result of blowing up $\partial B$ along $\partial B\cap B_0$. In other words:
\[ \widehat{\partial B}=\rho^{-1}(\partial B).
\]
It is a manifold with boundary that intersects transversely  $\eE$ along its boundary. Its boundary is diffeomorphic to $\partial S(\nu B_0)=S(\nu_{\partial B} (B_0\cap \partial B))$.

The  blow-up of $P$,  denoted  by $\hat{P}$ can be given  a direct description:
\[ \hat{P}=[-1,1]\times S(E),
\]
with projection map:
\[ \Bl:\hat{P}\ra P,\quad (t,v,b)\ra (t,v\sqrt{1-t^2},b), \;\;b\in B, v\in S(E_b).
\]
We let $\widehat{\partial P}:=\Bl^{-1}(\partial P)=[-1,1]\times S(E\bigr|_{\partial B})$.
 
The map $\varphi_0$ being transversal to $[0]$ (and to [$\infty$] which it does not touch) has a lift to a map:
\[ \hat{\varphi_0}:\hat{B}\ra \hat{P}
\]
and the same is true for $\varphi_t$ since flowing preserves transversality. We therefore get a family of commutative diagrams:
\begin{equation}\label{eq4.2} \xymatrix{ \hat{B}\ar[r]^{\hat{\varphi}_t} \ar[d]_{\rho}& \hat{P}\ar[d]^{\Bl}\\
    B\ar[r]^{\varphi_t}& P     }
\end{equation}
The lifts are described as follows:
\begin{equation} \label{eq7}\hat{\varphi}_t:=\left\{\begin{array}{cc}\hat{b}\ra \Bl^{-1}\circ\varphi_t\circ \rho(\hat{b}) & \mbox {for}\; \hat{b}\in \hat{B}\setminus \eE\\  \\
             (b,v)\ra \displaystyle\left(1,\frac{P_E\circ d_b\varphi_t(v)}{|P_E\circ d_b\varphi_{t}(v)|}\right)&\;\;\;\mbox{for}\; b\in B_0,\; v\in S(\nu_b B_0).\end{array}\right.
             \end{equation}
             
         In the last line of (\ref{eq7}) we use the fact that $T_{[0]}P$ is naturally isomorphic to $E\oplus TB$, therefore it makes sense to decompose $d_b\phi_t$ for $b \in B_0$  into its $E_b$ and $T_bB$ components. Hence $P_E$ stands for the projection onto the $E$ component which is normal to the blow-up locus.
         
      Each of the maps $\hat{\varphi}_t$ is an embedding.   Away from $\eE$, $\hat{\varphi}_t$ is essentially the section $\varphi_t$. Moreover, due to the transversality of $\varphi_t$ with $[0]$ we have that $P_E\circ d\varphi_t$ induces an isomorphism between the normal bundles of the blow-up loci. It should be clear that when restricted to the spherical normal bundle this gives again a diffeomorphism onto its image. 
      
  Notice that on $\hat{P}\subset \bR\times S(E)$ we can consider two  flows (vertical with respect to the natural projection $\hat{P}\ra S(E)$): 
\begin{itemize}
\item[I:] the lift of  $\Theta$ from the spherical bundle $S(\bR\oplus E)$: this is a flow with critical manifolds $\{\pm 1\}\times S(E)$; denote it by $\widehat{\Theta}$.
\item[II:] the flow without critical points generated by $-\frac{\partial}{\partial t}$; denote it by $\Psi$, properly speaking this is  a flow on $\bR\times S(E)$ but we will make use of it to flow  sets inside $\hat{P}$.
\end{itemize} 
It is not hard to check that $\hat{\varphi}_0$ is transverse to $\Psi$, meaning that the vector field $-\frac{\partial}{\partial t}$ is not of the type $d\hat{\varphi}_0(w)$. This is clear for points on the complement of $\eE$. For points on $\eE$, one notices that $\Imag\left(\hat{\varphi_0}\bigr|_{\eE}\right)\subset \{+1\}\times S(E)$ and the later set is transversal to the flow $\Psi$.

 Due to this transversality and the fact that $\Psi$ does not have critical points we infer  that when flowing $\hat{\varphi}_0(\hat{B})$ via $\Psi$ down to level $\{-1\}\times S(E)$\footnote{level which lies strictly after $\hat{\varphi}_0(\hat{B})$ in the direction of the flow} one gets  a manifold with corners $Q$ of dimension $n+1$ inside $\hat{P}$. This manifold is diffeomorphic by the Flowout Theorem to the product:
\[  \hat{B}\times I,
\]
where $I$ is an interval.   We argue that:
\begin{equation}\label{Qdef}Q=\overline{\bigcup_{t\in [0,\infty)}\hat{\Theta}_t(\hat\varphi_0(\hat{B}))}=\overline{\bigcup_{t\in[0,\infty)}\hat{\varphi}_t(\hat{B})}.\end{equation}
Notice first that $Q$ is a closed set. Moreover, $\bigcup_{t\in [0,\infty)}\hat{\Theta}_t(\hat\varphi_0(\hat{B}))$ is dense in $Q$. The only points that are in $Q$ and not in the former set are of type $(t,p)\in \hat{P}$ with $p\in \Imag\hat\varphi_0\bigr|_{\eE}$ and $t\in [-1,1]$. Now, each unstable trajectory of a point $p \in \Imag \varphi_0(B_0)\subset [0]$ is limit of trajectories corresponding to non-critical points that lie in $\Imag \varphi_0(B\setminus B_0)$. This follows again from transversality, by writting the flow in local coordinates around $p$ (the vector field is linearizable in suitable coordinates ).

We can conclude from (\ref{Qdef}) the following equality of currents:
\[ \lim_{t\ra \infty} \hat{\varphi}_*([0,t]\times \hat{B})=Q.
\]
This is because $\hat{\varphi}$ restricts a diffeomorphism from the $n+1$-dimensional manifold $(0,\infty)\times \hat{B}\setminus \partial \hat{B}$ to a subset of $Q$ of full $\mathcal{H}^{n+1}$-measure (the dim $n+1$ stratum of $Q$). The limit holds in fact in the mass topology of currents. It follows  that:
\[\lim_{t\ra \infty}d\hat{\varphi}_*([0,t]\times \hat{B})=d Q,
\]
the limit holding in the flat topology of currents. Apply now $\Bl_*$ to both sides to conclude that:
\begin{equation}\label{eq5.02}\lim_{t\ra \infty}d \varphi_*(\id\times \rho)_*([0,t]\times \hat{B})= \Bl_*(d Q).
\end{equation}
Since $\rho$ is a diffeomorphism from an open set of full measure to an open set of full measure we conclude that: $\rho_*\hat{B}=B$. Hence the left hand side of (\ref{eq5.02}) becomes
\[\lim_{t\ra\infty}d \varphi_*([0,t]\times B)=\lim_{t\ra\infty}(\varphi_t)_*B-(\varphi_0)_*B.
\]
The last relation being a consequence of Stokes on $[0,t]\times B$. Therefore:
\begin{equation}\label{lim1} \lim_{t\ra\infty}(\varphi_t)_*B=(\varphi_0)_*B+\Bl_*(d Q).
\end{equation}
Now since $Q$ is a compact manifold with corners we conclude that $d Q$ coincides with its codimension $1$ boundary:
\[ d Q=\hat{\varphi}_{\infty}(\hat{B}\setminus \partial \hat{B})+[-1,1]\times\hat{\varphi}_0(\partial\hat{B})-\hat{\varphi}_0(\hat{B}),
\]
where $\hat{\varphi}_{\infty}:\hat{B}\setminus \partial \hat{B}\ra \{-1\}\times S(E)$ is the pointwise limit $\hat{\varphi}_\infty(b)=\displaystyle\lim_{t\ra \infty}\hat{\varphi}_t(b)$. We conclude that:
\begin{equation}\label{lim2}\Bl_*(d Q)=\varphi_{\infty}(B\setminus B_0)+U(\varphi_0(B_0))-\varphi_0(B),
\end{equation}
where $U(\varphi_0(B_0))\ra \varphi_0(B_0)$ is the fiber bundle of unstable trajectories of the submanifold $\varphi_0(B_0)\subset [0]$. From (\ref{lim1}) and (\ref{lim2}) we conclude that:
\begin{equation}\label{lim3} \lim_{t\ra\infty}(\varphi_t)_*(B)=\varphi_{\infty}(B\setminus B_0)+U(\varphi_0(B_0)).
\end{equation}
In a completely analogous manner one proves that:
\begin{equation}\label{lim4} \lim_{t\ra\infty}(\varphi_t)_*(\partial B)=\varphi_{\infty}(\partial B\setminus \partial B_0)+U(\varphi_0(\partial B_0))
\end{equation}
Notice now that since the flow on $P$ is vertical we get by  a change of variables:
\begin{equation}\label{equ1}\mathscr{L}_{I}(\varphi_t^*\omega,\varphi_t^*\gamma)(\eta)=(\varphi_t)_*(B)(\omega\wedge\pi^*\eta)+(\varphi_t)_*(\partial B)(\gamma\wedge\pi^*\eta)
\end{equation}
We pass to limit in (\ref{equ1}) and using (\ref{lim3}) and (\ref{lim4}) we get:
\[ \lim_{t\ra\infty}\mathscr{L}_{I}(\varphi_t^*\omega,\varphi_t^*\gamma)(\eta)=\int_{\varphi_{\infty}(B\setminus B_0)}\omega\wedge \pi^*\eta+\int_{U(\varphi_0(B_0))}\omega\wedge \pi^*\eta+\]\[+\int_{\varphi_{\infty}(\partial B\setminus \partial B_0)}\gamma\wedge \pi^*\eta+\int_{U(\varphi_0(\partial B_0))}\gamma\wedge \pi^*\eta.
\]
Using the fact that $\varphi_{\infty}:B\setminus B_0\ra [\infty]$, resp. $\varphi_{\infty}\bigr|_{\partial B\setminus \partial B_0}$ are diffeomorphisms onto their images and that $\pi:{U(\varphi_0(B_0))}\ra B_0$ is a fiber bundle and using the adjunction formula we finally get (\ref{flim}). 

Part (b) follows in a similar manner from (\ref{lim3}) and (\ref{lim4}).
\end{proof}

\noindent
\emph{Proof of Theorem \ref{CGB}}

 For part (a) we take $\eta$ to be the family Pfaffian corresponding to the Levi-Civita connection of each fiber of $P\ra B$ constructed as follows. \footnote{This is the same proof we gave in \cite{Ci} for the boundaryless case, modulo computation of the boundary term.} 
 
 Let $d\oplus \pi^*\nabla$ be a connection on $\bR \oplus\pi^*E\ra P$ and let $\bR \oplus\pi^*E=\tau\oplus \tau^\perp$ be the  decomposition into the tautological line bundle and its orthogonal complement. Denote by $\nabla^{\tau^{\perp}}$ the connection on $\tau^{\perp}\ra P$ resulting by orthogonally projecting $\pi^*\nabla\oplus d$ onto $\tau^{\perp}$. Since $\tau^{\perp}\simeq VP$ one can interpret  $\nabla^{\tau^{\perp}}$ as the family Levi-Civita connection of the fiber bundle $P\ra B$.

Let  $\eta:=\Pf(\nabla^{\tau^{\perp}})$. We  compute via Theorem \ref{transfor}:
\[ \lim_{s\ra -\infty}\mathscr{L}_{II}(\varphi_s^*\eta)-\lim_{t\ra\infty}\mathscr{L}_{II}(\varphi_s^*\eta)= (\alpha(B,\eta),\alpha(\partial B,\iota^*\eta)),\]
where\[ \alpha(B,\eta)=B\wedge \widehat{\eta_{[0]}}-\left(\int_{P/B}\eta\right)[\varphi_0^{-1}(0)]-B\wedge\widehat{\eta_{[\infty]}}.
\]
On the other hand, due to Proposition \ref{heq} part (b) we infer that $\mathscr{L}_{II}(\varphi_s^*\eta)$ and $\mathscr{L}_{II}(\varphi_t^*\eta)$ for all $s,t$ represent the same  class. It is not hard to see that this stays true when $s\ra -\infty$ and $t\ra\infty$ since  $\eT^{\mathrm{II}}_t(\eta)$ have a limit as well which is
\[\lim_{\stackrel{t\ra\infty}{s\ra -\infty}}\eT^{\mathrm{II}}_t(\eta)=\left((\omega,\gamma)\ra -\int_{\bR\times B}\omega\wedge \varphi^*\eta,\int_{\bR\times \partial B}\gamma\wedge\varphi^*\eta\right).
\]
The convergence of the integrals is insured by the fact that the flow has a resolution as a compact manifold with corners: this is the closure of $\bigcup_{t\in \bR} \hat{\varphi}_t(B\setminus B_0)$.

Now,  $\int_{P/B}\eta=2$ by the Chern-Gauss-Bonnet theorem for the even dimensional  sphere with the round metric. Moreover $(\tau^{\perp}\bigr|_{[0]},\nabla^{\tau^{\perp}}\bigr|_{[0]})\simeq (E,\nabla)$ as pairs of bundles with connections. We have also $(\tau^{\perp}\bigr|{[\infty]},\nabla^{\tau^{\perp}}\bigr|_{[\infty]})\simeq (E,\nabla)$. One has to be careful here that when identifying $B$ with $[0]$ and $[\infty]$ respectively the orientations are different. In the end:
\[ \lim_{s\ra -\infty}\mathscr{L}_{II}(\varphi_s^*\eta)-\lim_{t\ra\infty}\mathscr{L}_{II}(\varphi_s^*\eta)=2\left(\Pf(\nabla)-s^{-1}(0), \Pf\left(\nabla\bigr|_{E\bigr|_{\partial B}}\right)-\partial s^{-1}(0)\right),
\]
from which the claim follows.

\vspace{0.5cm}

For part (b) one proceeds in a similar manner but with a bit of care. We set $\omega:=\Pf(\nabla^{\tau})$. Now, unfortunately there is no \emph{globally defined} transgression form $\gamma$ on $\partial P$, since  there is no  non-vanishing section on $\tau^{\perp}\ra S(\bR\oplus E)\bigr|_{\partial B}$  since $\tau^{\perp}_b\ra S(\bR\oplus E)_b$ is the tangent bundle of an even-dimensional sphere for every $b\in B$. 

However, one does not need the forms to be defined everywhere in order to apply Theorem \ref{transfor}. In our context we know that $\varphi_0$ is not vanishing along the boundary, hence there exists $U\subset \partial P$ an open neighborhood of $\infty$ which contains  $\Imag \varphi_t$ and $U\cap [0]\bigr|_{\partial P}=\emptyset$. One notices that it is enough for the form $\gamma$ to be defined in this neighborhood $U$, maintaining the pair $(\omega,\gamma)$ closed, in order to conclude that (\ref{flim}) still holds. 

Since we want to be able to take both limits $\displaystyle\lim_{t\ra \pm\infty}$ we will use two forms $\gamma$ depending on the case. As we already mentioned  $\tau^{\perp}$ restricted to $[0]$ or $[\infty]$ can be identified in a canonical manner with $E$. Hence there exist non-vanishing sections denoted $s_{[0]}$ and $s_{[\infty]}$ of the restrictions of $\tau^{\perp}$ to $[0]\bigr|_{\partial B}$ and $[\infty]\bigr|_{\partial B}$ by transferring the original section $s\bigr|_{\partial B}$ to $\tau^{\perp}$. Now extend $s_{[0]}$ and $s_{[\infty]}$ to nonvanishing sections on $\partial P\setminus [\infty]$ and $\partial P\setminus [0]$ respectively. Use $s_{[0]}$ and $s_{[\infty]}$ and $\nabla^{\tau^{\perp}}$ to define two transgression forms $\gamma_1\in \Omega^{2k-1}(\partial P\setminus [\infty])$ and $\gamma_2\in \Omega^{2k-1}(\partial P\setminus [0])$ as before such that $d\gamma_i=\Pf(\nabla^{\tau^{\perp}})$ where that makes sense.

Then we compute:
\[\lim_{s\ra -\infty}\mathscr{L}_I(\varphi_s^*\omega,\varphi_s^*\gamma_1)-\lim_{t\ra\infty}\mathscr{L}_I(\varphi_t^*\omega,\varphi_t^*\gamma_2)
\]
just as before and we are done.
\hfill\(\Box\)

\section{Thom isomorphisms and Chern-Gauss-Bonnet}
In this section, we show how the results we proved so far give us a different perspective on Thom isomophism. In deRham cohomology this celebrated result takes the following form:
\[ H^*_{\cpt}(E)\simeq H^{*-k}(B),
\] 
where $\pi:E\ra B$ is an oriented vector bundle of rank $k$. Both "arrows" are explicit.  The $\longrightarrow$ direction is integration over the fiber of differential forms, while the $\longleftarrow$ is:
\[ \omega\ra \tau \wedge \pi^*\omega,
\]
where $\tau\in \Omega^k_{\cpt}(E)$ is a closed form such that $\int_{E_b}{\tau}=1$ for all $b\in B$.

The purpose for the rest of this note is to discuss some other isomorphisms called by the same name. More precisely we want to make explicit the maps in the following diagram.
\begin{equation}\label{eq6.2} \xymatrix{ H^{i+k}(E, E_0)\ar[r]^{\iota^*}\ar[d]_{\mu} & H^{i+k}(\overline{DE},SE)\ar[r]^{LD ~ I}\ar@/_0.5pc/[d]_{\nu} & H_{n-i}(\mathscr{D}_*(\overline{DE}))\ar@/^0.5pc/[d]^{\pi_*}\\
             H^{i+k}_{\cpt}(E) \ar@/^0.5pc/[r]^{\int} &  H^{i}(B)\ar@/_0.5pc/[u]_{\nu^{-1}}\ar@/^0.5pc/[l]^{\tau \wedge\pi^*( \cdot)}\ar[r]^{PD\qquad} &  H_{n-i}(\mathscr{D}_*(B))\ar@/^0.5pc/[u]^{\iota_*}.}
\end{equation}
We used the standard notation $E_0:=E\setminus \{0\}$. The lower left horizontal isomorphisms are the classical Thom map and its inverse. The top horizontal map is the isomorphism induced by the inclusion of pairs:
\[ (\overline{DE},SE)\subset (E,E_0)
\]
The last vertical isomorphisms are induced by the homotopy equivalence $\pi:\overline{DE}\ra B$. The top right isomorphism is Lefschetz Duality while the bottom right isomorphism is Poincar\'e Duality:
\[ \omega\ra \left\{\eta\ra \int_B \omega\wedge\eta\right\}
\]
 One is left explaining what are the middle vertical arrows and the first vertical arrow.

We will start with the top-down middle vertical arrow. We claim that the map:
\begin{equation}\label{eq6.3} \Omega^*(\overline{DE},SE)\ra \Omega^{*-k}(B),\qquad \nu(\omega,\gamma)=\int_{\overline{DE}/B}\omega+\int_{SE/B}\gamma
\end{equation}
commutes up to a sign with the differential operators and therefore descends to a map in cohomology. To be more precise:
\begin{equation}\label{eq6.1}\nu(d(\omega,\gamma))=(-1)^kd\nu(\omega,\gamma).
\end{equation}
Indeed suppose $\deg{\omega}=i+k$. Then Stokes Theorem gives:
\[ \int_B\left(\int_{\overline{DE}/B}-d\omega\right)\wedge\eta+\int_B\left(\int_{SE/B}d\gamma+\iota^*\omega\right)\wedge \eta=\]\[=-\int_{\overline{DE}}d(\omega\wedge\pi^*\eta)+(-1)^{i+k-1}\int_{\overline{DE}}\omega\wedge \pi^*d\eta+\int_{SE}\iota^*\omega\wedge \iota^*\pi^*\eta+\]\[+\int_{SE}d(\gamma\wedge\iota^*\pi^*\eta)+(-1)^{i+k-1}\int_{SE}\gamma\wedge \iota^*\pi^*d\eta=\]\[=(-1)^{k}  \int_{B}(-1)^{i-1}\left(\int_{\overline{DE}/B}\omega+\int_{SE/B}\gamma\right)\wedge d\eta =(-1)^k\int_Bd\left(\int_{\overline{DE}/B}\omega+\int_{SE/B}\gamma\right )\wedge \eta.
\]
This proves (\ref{eq6.1}) and it is straighforward to check that the right square of (\ref{eq6.2}) commutes. It follows in particular that $\nu$ induces an isomorphism in cohomology.

Before we go on to describe the inverse of $\nu$, which is the more interesting part of this section, let us complete the diagram by saying what is the first vertical arrow. It turns out that this isomorphism appears implicitly in the proof that Nicolaescu gives in \cite{Ni} to the Chern-Gauss-Bonnet for closed manifolds. 

Assume that $E$ is endowed with a Riemannian metric and let $r:E\ra\bR$ be the radius function, i.e. $r(v)=|v|$. Consider $\rho:[0,\infty)\ra [0,\infty)$ to be a smooth function such that 
$\rho$ has compact support and $\rho\bigr|_{[0,1]}\equiv 1$. 

Define
\begin{equation}\label{eq6.4} \mu:\Omega^*(E,E_0)\ra \Omega_{\cpt}^*(E),\qquad \mu(\omega,\gamma)=\rho(r)\omega-\rho'(r)dr\wedge \gamma.
\end{equation}
Notice that 
\[ \mu(d(\omega,\gamma))=-\rho(r)d\omega-\rho'(r)dr\wedge (\omega+d\gamma)=-d(\rho(r)\omega)+d(\rho'(r)dr\wedge \gamma)=-d\mu(\omega,\gamma),
\]
hence we get a well-defined map in cohomology. For convenience $\rho$ will actually denote  $\rho\circ r$ from now on.
\begin{prop} The left square of diagram (\ref{eq6.2}) is commutative. 
\end{prop}
\begin{proof} Let $(\omega,\gamma)\in \Omega^{k+i}(E,E_0)$ be a closed pair. Then $d\omega=0$ and it follows from the homotopy equivalence $H^*(E)\simeq H^*(B)$ that:
\[ \omega=\pi^*\beta+d\theta,
\]
for some closed $\beta\in \Omega^{k+i}(B)$ and $\theta\in \Omega^{k+i-1}(E)$. Then for every \emph{closed form} $\eta\in \Omega^{n-i}(B)$:
\[\int_B\left( \int_{E/B}\rho \omega-\int_{E/B}d\rho\wedge\gamma\right)\wedge\eta =\int_{E} \rho\pi^*(\beta\wedge \eta)+\int_E\left(\rho d\theta-d\rho\wedge \gamma\right)\wedge\pi^*\eta=
\]
\[=\int_Ed(\rho\theta)\wedge \pi^*\eta-\int_E (d\rho\wedge (\theta+\gamma))\wedge \pi^*\eta=-\int_{DE^c}(d\rho\wedge (\theta+\gamma))\wedge \pi^*\eta,
\]
where $DE^c=\{v~|~|v|\geq 1\}$. We used Stokes for
 \[\int_Ed(\rho\theta)\wedge \pi^*\eta=\int_Ed(\rho\theta\wedge \pi^*\eta).\]
 Use Stokes again:
 \[-\int_{DE^c}(d\rho\wedge (\theta+\gamma))\wedge \pi^*\eta=-\int_{DE^c}d(\rho(\theta+\gamma)\wedge\pi^*\eta)=\int_B\left(\int_{SE/B}\rho(\theta+\gamma)\right)\wedge \eta=
 \]
 \[=\int_{SE}\theta\wedge\pi^*\eta+\int_B\left(\int_{SE/B}\gamma\right)\wedge \eta.
 \]
 Now use the fact that
 \[ \int_{DE}\omega\wedge\pi^*\eta=\int_{DE}d\theta\wedge\pi^*\eta=\int_{DE}d(\theta\wedge \pi^*\eta)=\int_{SE}\theta\wedge\pi^*\eta,
 \]
 to conclude that 
 \[ \int_{E/B}\mu(\omega,\gamma)=\nu(\iota^*(\omega,\gamma)).
 \]
 since both sides are equal when paired with a closed $\eta$. By Poincar\'e Duality the two sides have to be equal in $H^i(B)$.
\end{proof}
Notice that the commutativity of the diagram implies in particular that $\mu$ induces an isomorphism in cohomology.
\begin{remark} The only necessary requirements for $\rho$ in order for the left square of the Thom isomorphism diagram to be commutative is that $\rho$ have compact support, and be equal to $1$ in a neighborhood of $0$. Notice that it is necessary that $\rho$ be constant in a neighborhood of $0$ in order for $\mu$ to be well-defined, i.e. $\rho'(r)=0$ in a neighborhood of $0$. If we choose any other constant than $1$ then we will have to multiply by the same constant the isomorphism $\nu$. Nevertheless,  it is not essential that $\rho=1$ on $[0,1]$, since one can check easily that:
\[\int_{DE(r_1)/B}\omega+\int_{SE(r_1)/B}\gamma=\int_{DE(r_2)/B}\omega+\int_{SE(r_2)/B}\gamma,\qquad \forall \;r_1,r_2\in(0,\infty).
\]
for a closed pair $(\omega,\gamma)$.
\hfill\(\Box\)
\end{remark}
\vspace{0.3cm}
In order to understand what is $\nu^{-1}$ we can start by looking for a pair of forms in $(\overline{DE},SE)$ which is Lefschetz dual to the zero section in $\overline{DE}$. Notice that the zero section in $\overline {DE}$ corresponds via the last vertical arrow to the fundamental class of $B$, which corresponds via Poincar\'e duality to $1$ in $H^0(B)$. The answer to the question is then straightforward when the rank is even. The zero section is the zero locus of the tautological section $s^{\tau}:\overline{DE}\ra \pi^*E$. It is easily checked that $s^{\tau}$ is transversal to the zero section. Now if rank $k$ is even,  Theorem \ref{CGB} says that if $\nabla$ is a connection on $E$ then $(\pi^*\Pf(\nabla),-\TPf(\pi^*\nabla,s^{\tau}))$ is Lefschetz dual to $(s^{\tau})^{-1}(0)$.

In view of this we have:
\begin{prop}\label{evrTis} If rank of $E$ is even then the map:
\[ \nu^{-1}(\eta):= (\pi^*\Pf(\nabla)\wedge\pi^*\eta,-\TPf(\pi^*\nabla,s^{\tau})\wedge\pi^*\eta)
\]
is the inverse of $\nu$.
\end{prop}
\begin{proof} We compose with $\nu$ and notice that
\[ \int_{\overline{DE}/B}\pi^*\Pf(\nabla)\wedge\pi^*\eta=0,
\]
while 
\[ \int_{SE/B}\TPf(\pi^*\nabla,s^{\tau})=-1.
\]
The last relation follows either by a direct computation (see \cite{Ni}, Lemma 8.3.18) or by applying Chern-Gauss-Bonnet Theorem to (the tangent bundle of) the unit ball $D^{2n}\subset \bR^{2n}$. 
\end{proof}
Notice that the pair $(\pi^*\Pf(\nabla),-\TPf(\pi^*\nabla,s))$ makes sense as a closed pair on $(E,E_0)$. By applying $\mu$ and using the commutativity of the diagram we get thus the following explicit representative for a Thom form with compact support. The statement appears also in \cite{Ni} (see Lemma 8.3.18):
\begin{corollary}[Nicolaescu] \label{Niccor}The form $\rho(r)\pi^*\Pf(\nabla)+d(\rho\circ r)\wedge\TPf(\pi^*\nabla,s^{\tau})$ is a Thom form with compact support on the even rank, oriented Riemannian bundle $E$.
\end{corollary}

\section{The odd rank Thom isomorphism for pairs}

The challenge now is to find the odd dimensional counterpart of Proposition \ref{evrTis}.

We consider $\overline{DE}\subset S(\bR\oplus E)$ via the same inverse of the stereographic projection we used in the proof of Thereom \ref{CGB}.

 On $S(\bR\oplus E)$ we have a bundle of \emph{even rank}, namely $\bR\oplus \pi^*E $.\footnote{ We will use $\pi$ for all fiber bundle projections $S(\bR\oplus E)$, $E$, $SE$, $\overline{DE}$ to $B$. It should be clear from the context which one we mean.} This bundle is endowed with two connections:
\[\nabla^1:=d\oplus \nabla^{\tau^{\perp}} \qquad\mbox{and}\qquad \nabla^2:=d\oplus \pi^*\nabla ,
\]
where for $\nabla^1$ we use the tautological section of $\bR \oplus\pi^*E\ra S(\bR\oplus E)$ for the splitting. Since the Pfaffian of both $\nabla^1$ and $\nabla^2$ are zero as they both have parallel sections, we get by (\ref{hfor}) a closed form $\TPf(\nabla_1,\nabla_2)$ on $S(\bR\oplus E)$.

\vspace{0.3cm}

We claim that along the equator $SE\subset S(\bR\oplus E)$ the form $\TPf(\nabla_1,\nabla_2)$ is exact. We start with the following:
\begin{lemma}[The n-homotopy lemma.] Let $\Delta^m\subset \bR^m$ be the standard $n$-simplex, $B$ a closed oriented manifold and let $H:\Delta^m\times B\ra M$ be a smooth map (an $n$-homotopy). Then for all forms $\omega\in \Omega^k(M)$,  $k\geq m$ the following holds:
\[\int_{\Delta^m}H^*d\omega+(-1)^{m-1}d\int_{\Delta^m}H^*\omega=\int_{\partial \Delta^m}H^*\omega.
\]
\end{lemma}
\begin{proof} Let $\pi_2:\Delta^m\times B\ra B$ be the obvious projection. Stokes Theorem gives for $\eta\in \Omega^{n+m-k-1}(B)$:
\[\int_{\Delta^m\times B}d(H^*\omega\wedge \pi_2^*\eta)=\int_{\partial \Delta^m}H^*\omega\wedge \pi_2^*\eta=\int_{B}\left(\int_{\partial \Delta^m}H^*\omega\right)\wedge\eta.
\]
The left hand side can be written as:
\[\int_{B}\left(\int_{\Delta^m}H^*\omega\right)\wedge\eta+(-1)^k\int_{B}\left(\int_{\Delta^m}H^*\omega\right)\wedge d\eta=\]\[=\int_{B}\left(\int_{\Delta^m}H^*\omega\right)\wedge\eta+(-1)^{m-1}\int_Bd\left(\int_{\Delta^m}H^*\omega\right)\wedge\eta.
\]
\end{proof}

Suppose now $\pi:E\ra B$ is a Riemannian vector bundle and $\nabla^1,\nabla^2,\nabla^3\in\mathscr{A}(E)$ are $3$ \emph{metric compatible} connections. On $\pi_2^*E\ra \Delta^2\times B$ define the connection:
\begin{equation}\label{tina2} \tilde{\nabla}=\frac{d}{ds}+\frac{d}{dt}+\nabla^1+s(\nabla^2-\nabla^1)+t(\nabla^3-\nabla^1).
\end{equation}
Applying the homotopy formula to $\omega:=\Pf(\tilde{\nabla})$ and $H=\id_{\Delta^2\times B}$ we get:
\begin{lemma}\label{SSl} \[-d\int_{\Delta^2}\Pf(\tilde{\nabla})=\TPf(\nabla^1,\nabla^2)+\TPf(\nabla^2,\nabla^3)+\TPf(\nabla^3,\nabla^1).\]
\end{lemma}
\begin{proof} When pulling back to the edges of the $2$-simplex the bundle $\pi_2^*E$ and $\tilde{\nabla}$ one gets the corresponding connections as in (\ref{tina}).
\end{proof}
Although we will not need it we include the following
\begin{lemma}
\begin{equation}\label{eqn6} F(\tilde{\nabla})=ds\wedge(\nabla^2-\nabla^1)+dt\wedge(\nabla^3-\nabla^1)+\pi_2^*F(\nabla^{s,t}),
\end{equation}
where
\[ \nabla^{s,t}=\nabla^1+s(\nabla^2-\nabla^1)+t(\nabla^3-\nabla^1).
\]
\end{lemma}
\begin{proof}
To see that (\ref{eqn6}) holds write $\tilde{\nabla}=L+\Omega$ where $L=\frac{d}{ds}+\frac{d}{dt}+\nabla^1$ and
\[\qquad\qquad \Omega=s\Omega^{21}+t\Omega^{31},\qquad\quad \Omega^{i1}:=\nabla^i-\nabla^1,\;i=2,3.
\]
Then
\[F(\tilde{\nabla})=L^2+[L,\Omega]+\Omega\wedge\Omega.
\]
Now it is not hard to see that $L=\pi_2^*\nabla^1$ and therefore $L^2=\pi_2^*F(\nabla^1)$. On the other hand:
\[ [L,\Omega]=\left[\frac{d}{ds},\Omega\right]+\left[\frac{d}{dt},\Omega\right]+[\nabla^1,\Omega]=ds\wedge\Omega^{21}+dt\wedge\Omega^{31}+[\nabla^1,\Omega].
\]One checks easily that
\[ \pi_2^*F(\nabla^1)+[\nabla^1,\Omega]+\Omega\wedge\Omega=\pi_2^*F(\nabla^{s,t}),
\]
from which (\ref{eqn6}) follows.
\end{proof}

Lemma \ref{SSl} reminds one of a result by Simons and Sullivan. Namely if $\phi:S^1\ra  \mathscr{A}(E)$ is a smooth family of connections  then  there exists a closed form:
\[ \TPf(\phi):=\int_{[0,1]}\Pf(\tilde{\nabla}^{\phi}),
\]
where now $\tilde{\nabla}^{\phi}:=\frac{d}{dt}+\phi(t)$. Then  Proposition 1.6  from \cite{SS} says the following
\begin{prop}[Simons-Sullivan]\label{extrans} The form $\TPf(\phi)$ is exact. 
\end{prop}
\begin{proof}  Since $\mathscr{A}(E)$ is an affine space we can assume that there exists an extension of $\phi$ to a smooth map $\tilde{\phi}:D^2\ra\mathscr{A}(E)$.
On $D^2\times B$ consider the bundle $\pi_2^*E$ and the connection:
\[ \tilde{\nabla}^{\tilde{\phi}}:=d_{D^2}+\tilde{\phi}(z).
\]
Then
\[ \int_{S^1}\Pf(\tilde{\nabla}^{\tilde{\phi}})=-d\int_{D^2}\Pf(\tilde{\nabla}^{\tilde{\phi}})
\]
The left hand side is just $\TPf(\phi)$.
\end{proof}

\vspace{0.5cm}

We go back to our story. Consider the bundle $\bR \oplus\pi^*E$ over $SE$. 
It has a natural splitting:
\[ \bR \oplus\pi^*E= (\tau\oplus \bR)\oplus \tau^{\perp},
\]
where $\tau$ is determined by the tautological section of $\pi^*E\ra SE$. Consider the connection relative to this decomposition:
\begin{equation}\label{defnab3} \nabla^3:=d\oplus \nabla^{\tau^{\perp}}.
\end{equation}

We stress out that the trivial connection $d$ is defined on the trivial bundle $\tau\oplus \bR$ by using $s^{\tau}$ and $(0,1)$ for trivialization which is the same thing as declaring the sections parallel.

It is clear that $\Pf(\nabla^3)=0$. But, more interesting,  we also have:
\begin{lemma} \label{perspars}
\[ \TPf(\nabla^1,\nabla^3)=0= \TPf(\nabla^2,\nabla^3).
\]
\end{lemma}
\begin{proof} The reason for the vanishing of the two transgression classes is the existence of a persistent parallel section both for the family:
\begin{equation}\label{fafff} (1-t)\nabla^1+t\nabla^3
\end{equation}
as for the family
\begin{equation}\label{secafff} (1-t)\nabla^2+t\nabla^3.
\end{equation}
In the  case  of (\ref{secafff}), this is $s\equiv (0,1)$ while in the case of (\ref{fafff}) this is the tautological section $s^{\tau}$ of $\pi^*E\ra SE$. To see that this fact implies the vanishing of the transgression class take a look at the definition (\ref{transdef}). The section $(t,b)\ra s(b)$ is a parallel section of $p_2^*\pi^*E\ra [0,1]\times SE$ for the connection $\tilde{\nabla}$ and hence $\Pf(\tilde\nabla)=0$. 
\end{proof}

 By putting together  Lemma \ref{SSl} and Lemma \ref {perspars} we get:
\[ \TPf(\nabla^1,\nabla^2)\bigr|_{SE}=-d\int_{\Delta^2}\Pf(\tilde{\nabla}),
\] 
where $\tilde{\nabla}$ was defined at (\ref{tina2}). To keep the notation simple, let
 \begin{equation}\label{defsectra}\TPf(\nabla^1,\nabla^2,\nabla^3):=\int_{\Delta^2}\Pf(\tilde{\nabla})\end{equation} and call it a secondary transgression class.
We now prove:
\begin{theorem} Let $\rank{E}=k$ be odd.  The  pair $(-\TPf(\nabla^1,\nabla^2),-\TPf(\nabla^1,\nabla^2,\nabla^3))\in \Omega^{k}(\overline{DE},SE)$ is closed and is  Lefschetz dual to the zero section $B\hookrightarrow \overline{DE}$.
\end{theorem}
\begin{proof} The pair is closed due to the definition of $\TPf(\nabla^1,\nabla^2,\nabla^3)$.
We try to put ourselves in the conditions of Theorem \ref{transfor}. Let
\[  P:=S(\bR \oplus\pi^*E)\stackrel{\rho}{\longrightarrow} \overline{DE}\subset S(\bR\oplus E).
\]
be the fiberbundle projection. Over $P$ we have the bundle $\bR \oplus\rho^*\pi^*E$, the connections $\hat\nabla^2:=d\oplus\rho^*\pi^*\nabla$ and its cousin  $\hat{\nabla}^1:=\hat\nabla^{\tau}\oplus\hat{\nabla}^{\tau^{\perp}}$ according to the decomposition of $\bR \oplus\rho^*\pi^*E$ into its tautological line bundle and its orthogonal complement. The transgression form $\TPf(\hat{\nabla}^1,\hat{\nabla}^2)$ is closed. Moreover, the tautological section $s^{\tau}:\overline{DE}\ra P$ satisfies:
\[ (s^{\tau})^*\TPf(\hat{\nabla}^1,\hat{\nabla}^2)=\TPf({\nabla}^1,{\nabla}^2).
\]
This happens  $s^{\tau}$ pulls back $\bR \oplus\rho^*\pi^*E$ together with its tautological decomposition to $\bR \oplus\pi^*E$ and  its corresponding decomposition and $\TPf$ behaves as expected with respect to pull-back.

Now comes the more interesting part of defining $\hat{\nabla}^3$. Recall that we need a metric compatible connection. We will define $\hat{\nabla}^3$ on
  \[ F:=\bR \oplus\rho^*\pi^*E\bigr|_{S(\bR \oplus\pi^*E)\setminus ([0]\cup [\infty])}\]
   as follows: consider the plane bundle $\mathscr{P}\subset F$ generated by $\langle (1,0); s^{\tau}\rangle$. It is clear that $\mathscr{P}$ is trivializable. We fix such a trivialization by using as a basis the unit vectors $(1,0)$ and 
\begin{equation}\label{eq7.1} \frac{s^{\tau}-\langle s^{\tau},(1,0)\rangle (1,0)}{|s^{\tau}-\langle s^{\tau},(1,0)\rangle (1,0)|}.
\end{equation}
Via this trivialization we get a trivial connection on $\mathscr{P}$ denoted as always by $d$. Define:
\[\hat{\nabla}^3:=d\oplus \hat\nabla^{\perp}
\]
relative to the decomposition $F=\mathscr{P}\oplus \mathscr{P}^{\perp}$ where $\hat\nabla^{\perp}$ is the projection of $d\oplus \rho^*\pi^*\nabla$ onto $\mathscr{P}^{\perp}$. We will deal later with the unpleasant fact that $\hat{\nabla}^3$ is only defined away from the north and south pole section in $P$.

Notice that by the naturality of the constructions:
\[ \left(s^{\tau}\bigr|_{SE}\right)^*\TPf(\hat\nabla^1,\hat\nabla^2,\hat\nabla^3)=\TPf(\nabla^1,\nabla^2,\nabla^3).
\]
This makes sense since the image of $s^{\tau}\bigr|_{SE}$ is contained in the "equator"  of $S(\bR \oplus\pi^*E)$ where $(0,1)$ and the tautological section are actually orthogonal.

Notice also that  by Lemma \ref{SSl}:
\[-d\TPf(\hat\nabla^1,\hat\nabla^2,\hat\nabla^3)=\TPf(\hat\nabla^1,\hat{\nabla}^2)+\TPf(\hat{\nabla}^3,\hat{\nabla}^1),
\]
This suggests we  flow the closed pair:
 \[(\omega,\gamma):=(-\TPf(\hat\nabla^1,\hat{\nabla}^2)-\TPf(\hat{\nabla}^3,\hat{\nabla}^1),-\TPf(\hat\nabla^1,\hat\nabla^2,\hat\nabla^3)).\]

 Thinking of $\overline{DE}$ as a submanifold of $S(\bR\oplus E)$, let $s^{\tau}_t:\overline{DE}\ra P$ be $s^{\tau}$ flown by the fiberwise height flow to moment $t$.

We claim that
\[\lim_{t\ra \infty}\mathscr{L}_I((s^{\tau}_t)^*\omega,(s^{\tau}_t)^*\gamma)=[B].
\]
 where  $[B]$ is the zero section in $\overline{DE}$.
 
By Theorem \ref{transfor} we need to look at $\omega\bigr|_{[\infty]}$. But
 \begin{equation}\label{eq7.2} \omega\bigr|_{[\infty]}=0.
 \end{equation}
  This is because along the infinity (south pole) section as well as along the zero  (north pole) section the tautological decomposition of $\bR \oplus \rho^*\pi^*E$ coincides with the natural decomposition into $\rho^*\pi^*E$ and $\bR$, hence $\hat{\nabla}^1$ and $\hat{\nabla}^2$ coincide when everything is pulled back to these closed submanifolds of $P$. Similarly along the infinity section the connections $\hat\nabla^3$ and $\hat\nabla^1$ have a common parallel section namely $(0,1)=s^{\tau}$. Hence $\TPf(\hat{\nabla}^3,\hat{\nabla^1})\bigr|_{\infty}=0.$Taking into account that $\TPf$ commutes with pull-back we get (\ref{eq7.2}).
  
  In order to compute the residue
  \[ \int_{P/\overline{DE}}\omega,
  \]
  one uses again the observation in Proposition \ref{evrTis} to infer that
  \[\int_{S(\bR\oplus E_b)}\TPf(\hat\nabla^1,\hat{\nabla}^2)=-1.\] 

We will  deal with
\begin{equation} \label{2issues}\lim_{t\ra \infty}(s^{\tau}_t)^*\gamma\;\; \mbox{and}\;\;\int_{P/\overline{DE}}-\TPf(\hat{\nabla}^3,\hat{\nabla}^1)
\end{equation}
together. As we said before, $\gamma\bigr|_{\infty}$ is not apriori defined since $\hat{\nabla}^3$  makes no sense at $[\infty]$. To remedy this, we will blow-up the south and the north pole of $S(\bR \oplus\pi^*E)$ just as we did in the proof of Theorem \ref{transfor} and move everything to this space. Consider therefore the fiber bundle projection:
 \[\mu:[-1,1]\times S(\pi^*E)\ra \overline {DE},\qquad (t,v,w,b)\ra (w,b)\]
  and the smooth map (morphism of bundles over $\overline{DE}$):
\[ \qquad\qquad\mathscr{B}l: [-1,1]\times S(\pi^*E) \ra S(\bR \oplus\pi^*E),\]\[ \mathscr{B}l(t,v,w,b)= (t,v\sqrt{1-t^2},w,b),\qquad \forall v, w\in E_b, |v|=1, |w|\leq 1, t\in[-1,1].
\]
Since $\rho\circ \mathscr{B}l=\mu$ we get that $\mathscr{B}l^*\rho^*(\bR \oplus\pi^*E)$ is naturally isomorphic to $\bR\oplus \mu^*\pi^*E$.

On $\bR\oplus \mu^*\pi^*E\ra [-1,1]\times S(\pi^*E)$ there are two natural line bundles, namely $\bR\oplus 0$ and $\tilde{\tau}\subset \mu^*E$, the tautological bundle  induced by the tautological section of $\mu^*\pi^*E\ra S(\pi^*E) $, constant in $t$. These two line bundles span a \emph{global} plane subbundle $\tilde{\mathscr{P}}$ of  $\bR\oplus\mu^*\pi^*E$. On $\tilde{\mathscr{P}}$ consider the connection  $\nabla^{\tilde{\mathscr{P}}}$ which equals the trivial connections on $\bR\oplus\tilde{\tau}$ generated by the obvious sections and equals the projection of $d\oplus \mu^*\pi^*\nabla$ on the orthogonal complement of $\bR\oplus\tilde{\tau}$. One checks rather easily  that away from $\{\pm 1\}\times S(\pi^*E)$ the connection $\nabla^{\tilde{\mathscr{P}}}$ coincides with $\mathscr{B}l^*(\hat{\nabla}^3)$.

 There exists a third line subbundle $\mathscr{B}l^*\tau\subset \bR\oplus \mu^*\pi^*E$. Its fiber at $(t,v,w,b)$ is just
 \[\bR(t,v\sqrt{1-t^2})\subset \bR\oplus E_b.\]
 It is plain  to see that in fact $\mathscr{B}l^*{\tau}\subset \tilde{\mathscr{P}}$ since the fiber of $\tilde{\mathscr{P}}$ at $(t,v,w,b)$ is:
 \[ \langle (0,v),(1,0)\rangle.
 \]
 Now, just as before, the sections $(1,0)$ and $\mu^*s^{\tau}$ give rise to two connections  on $\bR\oplus\mu^*\pi^*E$.   These two connections coincide in fact with $\mathscr{B}l^*\hat{\nabla}^2$ and $\mathscr{B}l^*\hat{\nabla}^1$.

 Using $\mathscr{B}l^*\hat{\nabla}^1$, $\mathscr{B}l^*\hat{\nabla}^2$ and $\nabla^{\tilde{\mathscr{P}}}$, one defines a secondary transgression class $\TPf^{1,2,3}$ as in (\ref{defsectra}). Notice that via the tautological section $s^{\tau}: SE\ra S(\bR \oplus\pi^*E)$ seen as a section of $[-1,1]\times S(\pi^*E)$ by composing with $\mathscr{B}l^{-1}$,  the form $\TPf^{1,2,3}$ pulls back to  $\TPf(\nabla^1,\nabla^2,\nabla^3)$, again due to the naturality of the constructions. 
 
 Let $\hat{s}_t^{\tau}:SE\ra S(\pi^*E)\times [-1,1]$ be the section $\mathscr{B}l^{-1}\circ s^{\tau}_t$. Notice that when $t\ra \infty$, $\hat{s}_t^{\tau}\bigr|_{SE}$ converges smoothly to the section:
 \[ \hat{s}_{\infty}^{\tau}(w,b):=(w,-1,w,b)\in S(\pi^*E)\times \{-1\}.
 \]
One obviously has:
\begin{equation}\label{stau} (s^{\tau}_t)^*\gamma=(\hat{s}_t^{\tau})^*\TPf^{1,2,3}.
\end{equation}
Making  $t\ra \infty$ in (\ref{stau}) we are led to consider  $(\hat{s}_{\infty}^{\tau})^*\TPf^{1,2,3}$ which means we have to look at the three connections $\mathscr{B}l^*\hat{\nabla}^1$, $\mathscr{B}l^*\hat{\nabla}^2$ and $\nabla^{\tilde{\mathscr{P}}}$ when $\bR\oplus \mu^*\pi^E$ gets restricted to $\{-1\}\times S(\pi^*E)$. It is now straightforward to notice that along  $\{-1\}\times S(E)$, the connections $\mathscr{B}l^*\hat{\nabla}^1$ and $\mathscr{B}l^*\hat{\nabla}^2$ coincide and they also coincide with $d\oplus \mu^*\pi^*\nabla$. Moreover the constant section $(1,0)$ of $\bR\oplus \mu^*\pi^*E\ra \{-1\}\times S(\pi^*E)$   is also parallel for $\nabla^{\tilde{\mathscr{P}}}$. Hence all three connections under inspection have a common parallel section. This implies that the secondary transgression class they determine is $0$ since the Pfaffian of the auxiliary connection $\tilde{\nabla}$ in (\ref{tina2}) is zero as the later is easily seen to have  a parallel section.

 Hence $\displaystyle\lim_{t\ra \infty}(s^{\tau})_t^*\gamma=0$ which takes care of the first part of (\ref{2issues}).

For the second part, we observe that
\begin{equation}\label{mupush}\int_{P/\overline{DE}} -\TPf(\hat{\nabla}^3,\hat{\nabla}^1)=-\mu_*\TPf(\nabla^{\tilde{\mathscr{P}}},\mathscr{B}l^*\hat{\nabla}^1).
\end{equation}
We claim that the right hand side of (\ref{mupush}) is $0$ because of symmetry. Consider the bundle morphism:
\[ (\psi,\varphi):(\bR\oplus \mu^*\pi^*E, [-1,1]\times S(\pi^*E))\ra(\bR\oplus \mu^*\pi^*E, [-1,1]\times S(\pi^*E))
\]
\[\psi(t,u,s,w,v,b)= (-t,u,-s,w,v,b),\]\[ \varphi(s,w,v,b)=(-s,w,v,b),\qquad u,w,v\in E_b,\; |w|=1, |v|\leq 1
\]
It is relatively straighforward to check that the conditions of Proposition \ref{trandsym} are fullfilled and therefore:
\[\varphi^*\TPf(\nabla^{\tilde{\mathscr{P}}},\mathscr{B}l^*\hat{\nabla}^1)=\TPf(\nabla^{\tilde{\mathscr{P}}},\mathscr{B}l^*\hat{\nabla}^1).
\]
Hence,  integrating $\TPf(\nabla^{\tilde{\mathscr{P}}},\mathscr{B}l^*\hat{\nabla}^1)$ over $[-1,1]\times S(E_b)$ one gets $0$.
\end{proof}
\begin{remark} Technically speaking, in the proof above  one should work directly on $[-1,1]\times S(\pi^*E)$ as we did in the last part  rather than on  $S(\bR\times \pi^*E)$ as one has to justify also the existence of a homotopy formula "at infinity", in the same spirit to what was done in the proof of Theorem \ref{transfor}. The details are straightforward. 
\end{remark}

\begin{appendix}
\section{Characteristic forms and symmetry}
It is many times useful to know how certain symmetries of the connection reflect into symmetries of the characteristic forms associated to it.

Let $\pi:E\ra B$ be a vector bundle associated to a $G$-principal bundle $P\ra B$ where $G$ is a group via the action of $G$ on the vector space $V$. Let $Q$ be a $G$-invariant polynomial on $\End(E)$. Assume that $E$ is endowed with $G$-compatible connection $\nabla$, i.e. it is induced by a $G$-principal connection $\omega$ on $P$.

\begin{prop} \label{conact}
Let  $(\psi,\varphi):(E,B)\ra (E,B)$ be a bundle isomorphism that makes the diagram
\[ \xymatrix{E \ar[r]^{\psi}\ar[d]_{\pi} & E\ar[d]^{\pi}\\
 B\ar[r]^{\varphi} & B}
\] 
commutative such that
\begin{itemize}
\item[(i)] $\tilde\psi$ is a $G$-isomorphism.
\item[(ii)]$ \tilde{\psi}^{-1}\circ(\varphi^*\nabla)\circ\tilde{\psi}=\nabla.$
\end{itemize}
where $\tilde{\psi}:E\ra \varphi^*E$ is the induced isomorphism via the universal property of a cartesian product. Then
\[ \varphi^*Q(\nabla)=Q(\nabla).
\]
\end{prop}

The map $\tilde{\psi}:E\ra \varphi^*E$ is called a $G$-isomorphism if it arises from an isomorphism of principal bundles $\hat{\psi}:P\ra \varphi^* P$ by taking $\hat{\psi}\times \id_V$ and moding out the action of $G$.
\begin{proof} Straightforward by combining
\[Q(\varphi^*\nabla)=\varphi^*Q(\nabla)\]
with (ii) and the invariance of $Q$ under conjugation.
\end{proof}
Proposition \ref{conact} says that if $\nabla$ is invariant under the action of a $G$-bundle isomorphism pair $(\psi,\varphi)$ then any characteristic form associated to it is invariant under the action of $\varphi$.
\begin{example} If $B=S^{2n}$ is the even-dimensional sphere and  $A\in \SO(2n+1)$, then the Levi-Civita connection $\nabla$ is invariant under the action of the pair $(A,dA)$ hence the Pfaffian $\Pf(\nabla)$ satisfies:
\[ A^*\Pf(\nabla)=\Pf(\nabla).
\]
In other words, the Pfaffian is a rotationally invariant form on $S^{2n}$, hence it has to be a constant multiple of the volume form, constant that can be determined at any given point.
\end{example}
The same property can be investigated  for transgression forms. Assume for simplicity that $G= O(n)$. Consider two splittings
\[E=L_i\oplus L_i^{\perp},\qquad i=1,2
\]
where each line bundle $L_i$  is trivialized by a section $s_i:B\ra E$. Consider the connections induced by the splittings 
\[ \nabla^i=d\oplus \nabla^{L_i^{\perp}},\;\; i=1,2.
\] 
where $\nabla^{L_i^{\perp}}$ is the projection of $\nabla$ onto $L_i^{\perp}$.
\begin{prop}\label{trandsym} Suppose that in addition to the properties of Proposition \ref{conact}, the isomorphism $(\psi,\varphi)$ satisfies:
\begin{itemize}
\item[(iii)] $s_i\circ \varphi=\psi\circ s_i$, for all $b\in  B$ for both $i=1,2$.
\end{itemize}
Then
\[ \varphi^*\TPf(\nabla^1,\nabla^2)=\TPf(\nabla^1,\nabla^2).
\]
\end{prop}
\begin{proof} On one hand it follows directly from the definition that
\[ \varphi^*\TPf(\nabla^1,\nabla^2)=\TPf(\varphi^*\nabla^1,\varphi^*\nabla^2)=\TPf(\tilde\psi^{-1}(\varphi^*\nabla^1)\tilde\psi,\tilde\psi^{-1}(\varphi^*\nabla^2)\tilde\psi).
\]
On the other hand:
\begin{equation}\label{eqap1}\tilde{\psi}^{-1}(\varphi^*\nabla^i)\tilde\psi=\nabla^i.
\end{equation}
Let us check that (\ref{eqap1}) holds. From (iii) it follows that $\tilde{\psi}$ interweaves the trivializing sections for $L_i$ and $\phi^*L_i$ and hence:
\[ \tilde{\psi}^{-1}(\varphi^*d)\tilde{\psi}=d,
\]
on $L_i$, since $d$ is defined by $s_i$. One needs to prove that:
\begin{equation}\label{eqap2} \tilde{\psi}^{-1}\left({\varphi}^*P_{L_i^{\perp}}\nabla\right)(\tilde{\psi}\circ s)=P_{L_i^{\perp}}\nabla s,\qquad \forall s\in \Gamma(E).
\end{equation}
The next statement which follows from (iii)
\[ \tilde{\psi}\circ P_{L_i^{\perp}}=P_{\varphi^*L_i^{\perp}}\circ \tilde{\psi}
\]
together with (ii), imply that 
\[\tilde{\psi}\circ (P_{L_i^{\perp}}\nabla)_X(\tilde{\psi}^{-1}\circ t)=P_{\varphi^*L_i^{\perp}}(({\varphi}^*\nabla)_Xt),\;\;\forall t\in \Gamma(\varphi^*E).
\]
On the other hand it is straightforward to check that
\[ P_{\varphi^*L_i^{\perp}}(\varphi^*\nabla_Xt)=\varphi^*(P_{L_i^{\perp}}\nabla)_Xt
\]
Take now $s=\tilde\psi^{-1}\circ t$ and one gets (\ref{eqap2}). 
\end{proof}

\end{appendix}

\end{document}